\documentclass{amsart}
\usepackage{amsmath}
\usepackage{amsthm}
\usepackage{amssymb}
\usepackage{xfrac}
\usepackage{array}
\usepackage{mathtools}
\usepackage{stmaryrd}
\usepackage{dirtytalk}
\usepackage{bbm}
\usepackage{enumitem}
\usepackage{todonotes}
\usepackage{chngpage}

\newtheorem{prop}{Proposition}
\newtheorem*{prop*}{Proposition}
\newtheorem{definition}[prop]{Definition}
\newtheorem*{definition*}{Definition}
\newtheorem{theorem}[prop]{Theorem}
\newtheorem*{theorem*}{Theorem}
\newtheorem{thm}[prop]{Theorem}
\newtheorem{cor}[prop]{Corollary}
\newtheorem{calc}[prop]{Calculation}
\newtheorem{lemma}[prop]{Lemma}
\newtheorem{remark}[prop]{Remark}
\newtheorem*{remark*}{Remark}

\numberwithin{prop}{section}
\numberwithin{equation}{section}

\newcommand{\RR}{\ensuremath{\mathbb{R}}}
\newcommand{\CC}{\ensuremath{\mathbb{C}}}

\newcommand{\UU}{\ensuremath{\mathbb{U}}}

\newcommand{\GG}{\ensuremath{\mathbb{G}}}
\newcommand{\SSS}{\ensuremath{\mathbb{S}}}
\newcommand{\TT}{\ensuremath{\mathbb{T}}}
\newcommand{\BB}{\ensuremath{\mathbb{B}}}

\newcommand{\NN}{\ensuremath{\mathbb{N}}}
\newcommand{\ZZ}{\ensuremath{\mathbb{Z}}}
\newcommand{\QQ}{\ensuremath{\mathbb{Q}}}

\newcommand{\HH}{\ensuremath{\mathbb{H}}}

\newcommand{\Sub}[1]{\ensuremath{\mathrm{Sub}\left( #1 \right) }}
\newcommand{\Subd}[1]{\ensuremath{\mathrm{Sub}_\text{d}\left( #1 \right) }}

\newcommand{\Ad}[1]{\ensuremath{\mathrm{Ad} ( #1  ) }}

\newcommand{\hgt}[1]{\ensuremath{\mathrm{ht} ( #1  ) }}

\newcommand{\sse} {\ensuremath{s}}

\newenvironment{dedication}
        {\vspace{6ex}\begin{quotation}\begin{center}\begin{em}}
        {\par\end{em}\end{center}\end{quotation}}

\catcode`\@=11
\long\def\@savemarbox#1#2{\global\setbox#1\vtop{\hsize\marginparwidth 
  \@parboxrestore\tiny\raggedright #2}}
\catcode`\@=12

\begin{document}
\title{Effective discreteness radius of stabilisers for stationary actions}
\author{T.   Gelander \and A.   Levit \and G.A. Margulis}

\maketitle

\begin{dedication}
\vspace*{-0.5cm}

{Dedicated to Gopal Prasad on the occasion of his 75th birthday.}
\end{dedication}

\begin{abstract}
We prove an effective variant of the  Kazhdan--Margulis theorem generalized to stationary actions of semisimple groups over local fields:  the probability that the stabilizer of a random point admits a non-trivial intersection with    a small   $r$-neighborhood of   the identity is  at most $\beta r^\delta$ for some explicit constants $\beta, \delta > 0$ depending only the  group. This   is a consequence of a key convolution inequality. We deduce that    vanishing at infinity of injectivity radius implies finiteness of volume. Further applications are the compactness of the space of discrete stationary random subgroups and a novel proof of the fact that all lattices in semisimple groups   are weakly cocompact.
\end{abstract}




\section{Introduction}
Let $G$ be a semisimple real Lie group without compact factors. Recall 
 

\begin{theorem}[Kazhdan--Margulis \cite{kazhdan1968proof}]
\label{thm:Kazhdan Margulis}
There exists an identity neighborhood $V \subset G$ such that if $\Gamma$ is a discrete subgroup of $G$ then  $\Gamma^g \cap V = \{e\}$ for  some $g \in G$.
\end{theorem}

This  theorem  establishes the existence of a \emph{single} point $g\Gamma$ in the   probability space $G/\Gamma$ whose   stabilizer $\mathrm{Stab}_G(g\Gamma) = \Gamma^g$ intersects  trivially a given identity neighborhood. In fact, by \cite{gelander2018kazhdan}, the same conclusion holds true for \emph{most} points in any   probability space admitting a measure  preserving  action of $G$. Our main results make this statement quantitative, and apply more generally to   stationary measures rather than just invariant one.

   Let $K$ be a maximal compact subgroup of the Lie group $G$. Let $\mu$ be a bi-$K$-invariant probability measure on the  group $G$ whose support $\mathrm{supp}(\mu)$ is generating. 
      Recall that a   probability $G$-space $(Z,\nu)$ is called \emph{$\mu$-stationary} if $\mu * \nu = \nu$. In particular any $G$-invariant probability measure is such.

 Fix a left-invariant  Riemannian metric on the Lie group $G$ and denote by $\mathrm{B}_r$   the $r$-ball at the identity element of   $G$  with respect to this  metric.  
\begin{theorem}
\label{thm:main theorem}
There are   constants $\beta, \rho, \delta   > 0$   such that  every  $\mu$-stationary probability $G$-space $(Z,\nu)$ with $\nu$-almost everywhere discrete stabilizers $G_z$ satisfies 
\begin{equation}
\label{eq:main equation}
 \nu(\{z \in Z \: : \:  G_z \cap \mathrm{B}_r \neq \{\mathrm{id}_G\}   \}) \le  \beta r^\delta \quad \forall r < \rho.
 \end{equation}
\end{theorem}

The constants $\beta$ and $\rho$  depend on the choice of the   metric on the Lie group $G$  but are independent of the fixed probability measure $\mu$ and of the particular $\mu$-stationary $G$-space $(Z,\nu)$. 
 The constant $\delta$ depends only on the Lie group $G$ and admits the explicit lower bound   
\begin{equation}
 \delta  \ge  (3 \hgt{\mathfrak{g}} \dim_\RR G)^{-(\mathrm{rank}(K)+1) }
 \label{eq:final delta}
 \end{equation}
where  
 $\mathrm{rank}(K)$ is the   rank of the maximal compact subgroup $K$  and $\hgt{\mathfrak{g}} $ is the largest height\footnote{The family $\mathrm{A}_n$ has largest height $1$. The other classical families $\mathrm{B}_n$, $\mathrm{C}_n$,  $\mathrm{BC}_n$ and $\mathrm{D}_n$ all have largest height $2$. The exceptional semisimple Lie algebras have largest height $6$ at most.}
of any   positive root in the relative root system associated to $G$. 

The non-effective analogue of Theorem   \ref{thm:main theorem} is called \emph{weak uniform discreteness}. It  was previously established  in \cite{gelander2018kazhdan} for invariant probability measures by an abstract compactness argument in the space of invariant random subgroups. Clearly one   recovers  Theorem \ref{thm:Kazhdan Margulis} for any lattice $\Gamma$ in the Lie group $G$ as a special case of  \cite[Theorem 1.3]{gelander2018kazhdan}  (or of  Theorem \ref{thm:main theorem}) by applying it to the probability space $G/\Gamma$ with respect to any sufficiently small radius $ r > 0$.


We remark that assuming the measure $\nu$ is invariant and non-atomic,  if the  Lie group $G$ is simple, or more generally   if the $G$-action is irreducible (i.e. every non-central normal subgroup acts ergodically), then $\nu$-almost every stabilizer subgroup will be discrete \cite{abert2017growth}.

While the conclusion of Theorem \ref{thm:main theorem} is  indeed  independent of the    choice of the   probability  measure $\mu$, our proof   relies  on the   properties of a specific  and  carefully chosen  such probability measure,   see Theorem \ref{thm:inequality} below.

 \begin{remark}
The constant $\rho$  appearing in the statement of Theorem \ref{thm:main theorem}  is  strictly speaking redundant, in the sense that  Equation (\ref{eq:main equation})  holds for \emph{all} values of  $r > 0 $ provided that  the constant $\beta$ is replaced by   $\max \{ \beta,   \rho^{-\delta} \}$. We prefer to keep our sharper statement as  is,  since   the constant $\rho$ plays an essential role in the   proof.
\end{remark}

\subsection*{The positive characteristic case}

 Let $k$ be a non-Archimedean  local field
and $\GG$ a connected simply-connected semisimple $k$-algebraic linear group without $k$-anisotropic factors.  Take $G = \GG(k)$ so that $G$ is a $k$-analytic group. 

If $\textrm{char}(k) = 0$  then $G$ has no small discrete subgroups \cite[Part II, Chapter V.9, Theorem 5]{serre2009lie}, namely there is a compact open subgroup $U \le G$ such that every discrete subgroup $\Gamma $ satisfies $\Gamma \cap U = \{\mathrm{id}_G\}$.  This fact is already much stronger than any potential analogue of     Theorems \ref{thm:Kazhdan Margulis}  or \ref{thm:main theorem}. 

Assume therefore that  $\textrm{char}(k)$ is positive and is a good prime\footnote{Good primes are discussed e.g. in  \cite[\S4]{springer1970conjugacy}. The only primes which  fail to be  good for some semisimple group $\GG$ are $2,3$ and $5$.} for $\GG$.
Our main result  naturally generalizes to the $k$-analytic group $G$, answering   in  the positive \cite[Question 4.3]{gelander2018kazhdan} and extending Theorem \ref{thm:Kazhdan Margulis}  to this setting.

We now state our main result in the positive characteristic case. Let $\mathcal{O}$ be the ring of integers of the non-Archimedean local field $k$ and $\mathfrak{m}$ be the maximal ideal of $\mathcal{O}$. Let $ \GG(\mathfrak{m}^i)$ denote the congruence subgroup\footnote{The exact definition of each congruence subgroup $G(\mathfrak{m}^i)$ depends (up to finite index) on the particular matrix realization of the $k$-algebraic group $\GG$.} of the compact group $\GG(\mathcal{O})$ modulo the ideal $\mathfrak{m}^i$ for every $i \in \NN$. The subgroups $\GG(\mathfrak{m}^i)$  form a basis of      identity neighborhoods for the topology of $G$. 

Let $\mu$ be a  bi-$\GG(\mathcal{O})$-invariant probability measure on the $k$-analytic group $G$ such that the  support of $\mu$   generates the group $G$.


\begin{theorem}
\label{thm:main theorem in positive char}
There are  constants  $\beta, \rho, \delta > 0$    such that every $\mu$-stationary probability $G$-space $(Z,\nu)$ with $\nu$-almost everywhere discrete stabilizers $G_z$ satisfies
\begin{equation}
\label{eq:main result positive char}
 \nu(\{z \in Z \: : \:  G_z \cap \GG(\mathfrak{m}^i) \neq \{\mathrm{id}_G\}   \}) \le  \beta \left| \sfrac{\mathcal{O}}{\mathfrak{m}}\right|^{-\delta i}
 \end{equation}
for all $i \in \NN$ with $\left| \sfrac{\mathcal{O}}{\mathfrak{m}}\right|^{-i} < \rho $.
\end{theorem}

The constants $\beta, \rho$  and $\delta$ are independent of the particular probability measure $\mu$ as well as of the $\mu$-stationary probability $G$-space   $(Z,\nu)$. The constant $\delta$  admits the explicit lower bound
  \begin{equation}
\delta  \ge \frac{1}{  \hgt{\mathfrak{g}}^2  \dim^5_k G}.
  \end{equation}

Alternatively, one may consider the norm  $\|\cdot\|_\mathfrak{m}$ on the   group $\GG(\mathfrak{m})$ given  by 
\begin{equation}
\label{eq:p norm}
 \|g\|_\mathfrak{m} = 
 \inf  \: \{ \left| \sfrac{\mathcal{O}}{\mathfrak{m}}\right|^{- i}   \: : \: \text{$i \in \NN$ and $g \in \GG(\mathfrak{m}^i)$}    \} 
 \end{equation}
for all elements $g \in \GG(\mathfrak{m})$ and reformulate  Equation (\ref{eq:main result positive char})   in terms of $r$-balls with respect to $\|\cdot\|_\mathfrak{m}$ in a manner analogous to Equation (\ref{eq:main equation}).

Just as in the Archimedean case,   if $(Z,\nu)$ is a non-atomic irreducible probability measure preserving $G$-space then $\nu$-almost every stabilizer subgroup is discrete   \cite[Theorem 1.9]{gelander2018invariant}.

We remark that while an    analog of the Kazhdan--Margulis theorem over positive characteristic local fields was proved     for products of rank-one groups in \cite{raghunathan1989discrete,lubotzky1990lattices,lubotzky1991lattices} and  for simply connected Chevalley groups of any rank in \cite{golsefidy2009lattices}, the   general case has not been available in the literature until now.


\subsection*{The key inequality}
A fundamental  part of this work is the inequality stated in   Theorem \ref{thm:inequality} below. The above results 
Theorems \ref{thm:main theorem} and \ref{thm:main theorem in positive char} are both   derived 
in a  rather  straightforward manner from this inequality, which is of independent interest and admits various other applications.

To fix notations, let $k$ be a local field and $\GG$  a connected simply-connected semisimple $k$-algebraic linear group.  Denote $G = \GG(k)$ so that $G$ is    $k$-analytic. Let $K$ be a maximal compact subgroup of $G$ (if $k$ is non-Archimedean assume moreover the subgroup $K$ is  good).

Let $\mu_{\sse}$   be the  following   bi-$K$-invariant probability measure on the $k$-analytic group $G$ 
 \begin{equation}
 \mu_{\sse}=\eta_K * \delta_{\sse} * \eta_K
 \end{equation}
  where $\eta_K$ is the normalized Haar probability measure of the  compact  group $K$ and $\delta_{\sse}$ is an atomic probability measure supported on some sufficiently expanding regular semisimple element $\sse \in G$, see \S\ref{sec:key-inequality} for details. Note that $\mu_{\sse}$ is compactly supported.

To state the key inequality    we consider the \emph{discreteness radius}   function $\mathcal{I}_G(\Gamma)$ defined for every  discrete subgroup $\Gamma$  of the $k$-analytic group $G$  by
\begin{equation}
 \mathcal{I}_G(\Gamma)= \sup \{ 0 \le r \le \rho \: : \: \Gamma \cap \mathrm{B}_r = \{\mathrm{id}_G \} \} 
 \end{equation}
 for some  constant  $\rho > 0$ (determined in \S\ref{sec:Zass}). Here $\mathrm{B}_r = \exp( \{ X \in \mathfrak{g} :  \|X\| \le r \})$ with respect to some norm $\|\cdot\|$ on the Lie algebra $ \mathfrak{g} = \mathrm{Lie}(G) $ in the  Archimedean case and $\mathrm{B}_r = \{ g \in \GG(\mathfrak{m}) \: : \|g\|_\mathfrak{m} \le r \}$ in the non-Archimedean case.


%


  
%
%
  
%
%


\begin{thm}[The key inequality]
\label{thm:inequality}
There are constants $0 < c <1$  and $b, \delta  > 0$   such that  every discrete subgroup $\Gamma$ of the group $G$ satisfies
\begin{equation}
\label{eq:convolution with f_G}
\int_G \mathcal{I}_G^{-\delta }(\Gamma^g) \: \mathrm{d}\mu_{\sse}(g)  \le c   \mathcal{I}_G^{-\delta }(\Gamma) + b. 
 \end{equation}
 \end{thm}

 The constant $\delta$   depends only   on the analytic group $G$ in question. Moreover $\delta$ and $\rho$  coincide  with the eponymous constants appearing in Theorems \ref{thm:main theorem} and \ref{thm:main theorem in positive char}.

%


\subsection*{The space of discrete  stationary random subgroups is compact}

We consider the space $\Sub{G}$ of  all closed subgroups of the group $G$ equipped with the Chabauty topology where  $G$ is a $k$-analytic group   as  above.
This is a compact space on which   $G$ acts continuously by conjugation.
Let $\Subd{G}$ be the subspace\footnote{If the local field $k$ is Archimedean then the subspace $\Subd{G}$ is Chabauty open.} of $\Sub{G}$  consisting of     discrete subgroups. It is   $G$-invariant but   no longer compact.  
A \emph{discrete random subgroup} of $G$ is a Borel probability measure $\nu$ on the Chabauty space $\Sub{G}$ satisfying $\nu(\Subd{G}) = 1$.

Any   probability $G$-space $(Z,\nu)$ with $\nu$-almost surely discrete stabilizers  determines a discrete random subgroup of $G$ by pushing forward the measure $\nu$ via the stabilizer map $Z\to\text{Sub}(G),~z\mapsto G_z$. Vice versa, every discrete random subgroup of $G$ can be realized in this way with respect to some  probability $G$-space\footnote{See \cite[Theorem 2.6]{abert2017growth} for a proof of the analogue fact for invariant random subgroups. The same proof applies to any locally compact group $G$ and any measure $\nu$ on $\Sub{G}$ as long as   $\nu$-almost every subgroup is  unimodular.}. 

This alternative view point allows Theorems \ref{thm:main theorem}, \ref{thm:main theorem in positive char} and \ref{thm:inequality} to be reformulated as results about discrete  stationary random subgroups. In particular, the  weak uniform discreteness of the discrete stationary random subgroups can be stated  as a compactness result. 


\begin{cor}\label{cor:compact} 
Let $\mu$ be a bi-$K$-invariant probability measure on the group $G$. Then the set of $\mu$-stationary probability measures on the space $\Subd{G}$ is compact in the weak-$*$ topology.
\end{cor}

%

We remark that in the Archimedean case  it is possible to go in the opposite direction and deduce weak uniform discreteness from the weak-$*$ compactness of the set of discrete invariant random subgroups, see \cite{gelander2018kazhdan}.



\subsection*{Evanescence  and finiteness of volume}

Let $M$ be a Riemannian manifold of non-positive sectional curvature.
 The following notion is introduced in \cite{Ballmann-Gromov-Schroeder}: the manifold $M$ has $\text{InjRad}\to 0$ (hereafter \emph{evanescent})  
 if its injectivity radius   $\textrm{Inj-Rad}_M$ vanishes at infinity. In other words $M$ is evanescent if and only if its  $\varepsilon$-thick part $M_{\ge\varepsilon} = \{ x \in M \, : \, \textrm{Ind-Rad}_M(x) \ge \varepsilon\}$ is compact for every $\varepsilon>0$. This is equivalent to the finiteness of the $\varepsilon$-essential  volume\footnote{We use $\varepsilon$-essential volume in the sense  of \cite[\S12.B]{Ballmann-Gromov-Schroeder}.} for every $\varepsilon > 0$.  
 
 Gromov  proved that an evanescent analytic manifold of bounded non-positive sectional curvature  is diffeomorphic to the interior of a compact manifold with boundary \cite[Theorem 2]{Ballmann-Gromov-Schroeder}.

It is clear that finite volume   implies evanescence. The converse of this fact for locally symmetric spaces is a consequence of our key inequality (Theorem \ref{thm:inequality}).
 
\begin{thm}
\label{thm:evanescence}
A locally symmetric manifold is evanescent if and only if it has finite volume.
\end{thm}


\subsection*{Weak cocompactness of lattices}

A lattice $\Gamma$ in a locally compact group $H$  is \emph{weakly cocompact}  if the quasiregular unitary representation of $H$ in the Hilbert space $L^2_0(H/\Gamma)$ does not admit almost invariant vectors, or in other words   the $H$-space $H/\Gamma$ has spectral gap.  

Cocompact lattices are   weakly cocompact \cite[III.1.8]{margulis1991discrete}.  
It was asked in \cite[III.1.12]{margulis1991discrete} whether all lattices are weakly cocompact. This is true for Lie groups \cite{bekka1998uniqueness,bekka2010spectral},   simple algebraic groups over non-Archimedean local fields  \cite[Theorem 1]{bekka2011lattices} and   groups with Kazhdan's property (T). 

Some locally compact groups however do admit non-weakly cocompact lattices. A first explicit example in the group of tree automorphisms was constructed  in \cite[Theorem 2]{bekka2011lattices}. 
A lattice in an amenable second countable locally compact
group is  cocompact if and only if it is  weakly cocompact \cite[Proposition 1.8]{BCGM}. Thus, the non-uniform lattices constructed in \cite[Theorem 1.11]{BCGM} provide  further non-weakly cocompact examples.


We  obtain the following result relying on   the key inequality (Theorem \ref{thm:inequality}) and on the methods developed towards its proof.

\begin{theorem}
\label{thm:weakly cocompact}
Let $k$ be a local field. Let    $\GG$ be a connected simply-connected semisimple $k$-algebraic group without $k$-anisotropic factors.  Then every lattice in the locally compact  group  $\GG(k) $ is weakly cocompact.
\end{theorem}
 
The conclusion of Theorem \ref{thm:weakly cocompact} has been previously known in most  (but not all) cases. Its  novelty  lies in obtaining a  unified and a direct proof.

Interestingly Theorem   \ref{thm:weakly cocompact} is deduced   in a rather straightforward manner from  a stronger and a more general statement,  namely Theorem \ref{thm:uniform local spectral gap at infinity} below. Roughly speaking, this last theorem says that the norm of   a convolution operator   restricted to the "thin part" of any probability measure preserving action of the group $\GG(k)$ is  bounded away from one uniformly over all such actions. This result is novel even in the presence of Kazhdan's  property (T).

\section{Averaging operators}
\label{sec:margulis functions}

Let $G$ be a locally compact second countable group and $\mu$   a probability measure on  $G$. Let $Z$ be a  locally compact  topological space admitting a continuous action of the group $G$.  

We study the averaging  operator $A_\mu$ acting on   continuous functions  on the   space $Z$.  This operator is used to bound the $\nu$-measure of super-level sets of continuous functions satisfying a certain key inequality, where $\nu$ is  any $\mu$-stationary measure  on the space $Z$. 

The current discussion uses ideas  from \cite{eskin1998upper, margulis2004random}. 
See also the works \cite{eskin2001asymptotic, eskin2002recurrence,athreya2006quantitative,einsiedler2009entropy,margulis2010quantitative,buterus2010distribution, benoist2012random,kadyrov2014entropy,eskin2015isolation,han2016asymptotic,mohammadi2020isolations} for related ideas and further applications.

\subsection*{Bounds on super-level sets}

Consider the averaging operator $A_\mu$ corresponding to the probability measure $\mu$. That is, given  any continuous  function $F : Z \to \left[0,\infty\right]$ we define the function $A_\mu F : Z \to \left[0,\infty\right]$ by
\begin{equation}
\label{eq:convolution operator}
   A_\mu F(z) = \int_G F(gz) \, \mathrm{d}\mu(g) \quad \forall z \in Z.
 \end{equation}

 Let $\nu$ be a  $\mu$-stationary Borel probability measure on   $Z$.  
\begin{lemma}
\label{lem:bound on contracting function}
Assume that the space $Z$ is $\sigma$-compact. Let $F : Z \to (0,\infty)$ be a continuous function.
 If   there are constants $0 < c < 1$ and $b > 0$ such that 
\begin{equation}
\label{eq:contraction F topological}
 A_\mu F (z) \le c F (z) + b
 \end{equation}
  then  
\begin{equation}
\label{eq:bound on prob topological}
 \nu (\{z \in Z \: : \: F(z) \ge M\}) \le \frac{b}{(1-c)M} \quad \forall M > 0.
 \end{equation}
\end{lemma}
\begin{proof}
 Consider the sequence of probability measures  $\mu_m =   \frac{1}{m} \sum_{i=1}^{m } \mu^{*i}$ on the   group $G$   for all $m \in \NN$. Note that the probability measure $\nu$ is $\mu_m$-stationary for all $m \in \NN $, namely $\mu_m * \nu = \nu$.

Since the space $Z$ is assumed to be $\sigma$-compact it is possible to write $Z = \bigcup_{n\in\NN} Z_n$ where each $Z_n$ is a compact subspace and $Z_n \subset Z_{n+1}$ for all $n\in \NN$. 
Fix an index $n \in \NN$ and denote   $\nu_n = \nu_{|Z_n}$.  Decompose the probability  measure $\nu$ as
\begin{equation}
\label{eq:decomposition measures}
 \nu = \mu_m * \nu = \mu_m * \nu_n + \mu_m * (\nu - \nu_n).
 \end{equation}
 
Up to passing twice to a subsequence,  we may assume that $\mu_m * \nu_n\xrightarrow{m \to \infty} \nu_n^\infty$ and $\mu_m * (\nu - \nu_n) \xrightarrow{m \to\infty} \eta_n^\infty$ in the weak-$*$ topology for some  positive $\mu$-stationary measures $\nu_n^\infty$ and $\eta_n^\infty$    on the   space $Z$. Generally speaking, the Portmanteau theorem says that $\nu_n^\infty(Z) \le \nu_n(Z)$ and $\eta_n^\infty(Z) \le (\nu-\nu_n)(Z) = 1-\nu_n(Z)$.  Note however that passing to the limit in Equation (\ref{eq:decomposition measures})  gives $\nu = \nu_n^\infty + \eta_n^\infty$. Therefore it must  be the case  that $\nu_n^\infty(Z) = \nu_n(Z)$ and $\eta_n^\infty(Z) = 1- \nu_n(Z)$.  In particular
\begin{equation}
\label{eq:inf is 0}
\inf_n \eta_n^\infty(Z)  = 1- \sup_n \nu_n^\infty(Z)= 1- \sup_n \nu_n (Z) = 0.
\end{equation}

Clearly $\sup_{Z_n} F < \infty$ as the subspace $Z_n$ is compact and the function $F$ is continuous.  Iterating Equation (\ref{eq:contraction F topological}) gives  
\begin{equation}
A_\mu^i F (z) \le \frac{b}{1-c} + c^i F(z)    \le  \frac{b}{1-c} + \sup_{Z_n} F \cdot c^i 
\end{equation}
for all points $z \in Z_n$. Therefore
\begin{align}
\begin{split}
\int_Z F \mathrm{d}\nu^\infty_n &= 
\lim_{m \to\infty} \int_Z F \; \mathrm{d}  (\mu_m * \nu_n) = 
\lim_{m\to\infty} \frac{1}{m} \sum_{i=1}^m \int_Z A_\mu^i F \; \mathrm{d} \nu_n \le \\
&\le \frac{b}{1-c} + \sup_{Z_n} F \cdot \lim_{m\to\infty} \frac{1}{m} \sum_{i=1}^m c^i   = \frac{b}{1-c}.
\end{split}
\end{align}
 Markov's inequality applied with respect to the  measure $\nu_n^\infty$ for each $n \in \NN$ and combined with Equation (\ref{eq:inf is 0}) implies that
\begin{multline}
\begin{split}
 \nu (\{z \in Z \: : \: F(z) \ge M\}) = \inf_n\; (\nu_n^\infty + \eta_n^\infty)(\{z \in Z \: : \: F(z) \ge M\})   \le \\
\le \frac{b}{(1-c)M}  + \inf_n \eta_n^\infty(Z)   
 = \frac{b}{(1-c)M}  
 \end{split}
\end{multline}
for all $M > 0$ as required.
\end{proof}

\begin{remark*}
It is possible to prove an analog of  Lemma \ref{lem:bound on contracting function}   in the measurable setting  and without assuming the continuity of the function $F$. This alternative   relies on the von Neumann ergodic theorem and requires     the averaging operator $A_\mu$ to admit no non-trivial invariant vectors. 
\end{remark*}

\subsection*{From expansion to contraction}
Let $f : Z \to \left(0,\infty\right)$ be a non-negative continuous function.
  Our    goal  is to show that, under suitable assumptions,   if $f$ is pointwise expanded by the averaging operator $A_\mu$ on some subset of large measure, then the function $F = f^{-\delta}$  is  pointwise contracted by $A_\mu$ for  some sufficiently small exponent $\delta > 0$.  We begin with an elementary calculation.

\begin{calc}
	\label{calc:calculus}
	Let $a_1 > 1 >a_2 > 0 $ and   $0 < p < 1$. The real function $\varphi$ given by
	\begin{equation}
	\label{eq:phi}
	\varphi(\delta) = p a_1 ^{-\delta} + (1-p) a_2^{-\delta}
	\end{equation}
has a global minimum point at 
	\begin{equation}
\label{eq:delta global min}
\delta_0 = \delta_0(a_1,a_2;p) = -\frac{\ln\left(- \frac{1-p}{p} \frac{\ln a_2}{\ln a_1} \right)}{\ln\left(\frac{a_1}{a_2}\right)}.
\end{equation}
Moreover if
	\begin{equation}
	\label{eq:balance}
	(1-p) \ln \frac{1}{a_2} < p \ln a_1
	\end{equation}
	then  $\delta_0 > 0$ and $\varphi(\delta_0) < 1$.
\end{calc}

\begin{proof}
Note that $\varphi(0) = 1$ and $\lim_{\delta \to \pm \infty} \varphi(\delta) = +\infty$.
	The derivative $\varphi'$ of the function $\varphi$ is given by
	\begin{equation}
	\label{eq:phi prime}
	\varphi'(\delta) = - p \ln a_1 \cdot a_1^{-\delta} - (1-p) \ln a_2 \cdot a_2 ^{-\delta}.
	\end{equation} 
A real number $\delta_0$ is a critical point of the function $\varphi$ if and only if $\varphi'(\delta_0) = 0$. This condition is equivalent to 
	\begin{equation}
	\left(\frac{a_1}{a_2}\right)^{-\delta_0} = - \frac{1-p}{p} \frac{\ln a_2}{\ln a_1} > 0.
	\end{equation}
	Therefore $\varphi$ has a unique critical point  $\delta_0 = \delta(a_1,a_2;p)$ whose exact value is given by Equation (\ref{eq:delta global min}). As $a_1 / a_2 > 1$ this critical point  satisfies $\delta_0 > 0$ if and only if 
	$$ - \frac{1-p}{p} \frac{\ln a_2}{\ln a_1} < 1 \quad \Leftrightarrow \quad  -(1-p) \ln a_2 < p \ln a_1.$$
	Moreover, in this case $\varphi(\delta_0) < \varphi(0) = 1$.
\end{proof}

\begin{prop}
\label{prop:from expansion to contraction} 
	Let $a_1 > 1 >a_2 > 0$ and $0 < p < 1$ be such that Equation (\ref{eq:balance}) holds. Let $\rho_0 > 0$ be such that   every point $z \in Z$ satisfies
\begin{enumerate}
\item if $f(z) < \rho_0$ then the subset
$$L_z =  \{g \in G \: : \: f(gz) \ge a_1 f(z)  \}$$
satisfies $\mu(L_z) \ge p$ and
\item $f(gz) \ge a_2 f(z)$ holds for $\mu$-almost every element $g \in G$.
\end{enumerate}
Then there are constants $0 < c < 1$ and $b > 0$ such that
\begin{equation}
\label{eq:conclusion on contraction}
A_\mu f ^{-\delta} \le c f^{-\delta} + b
\end{equation} 
where $ \delta = \delta_0(a_1,a_2;p)$ is as given  by Equation (\ref{eq:delta global min}).
\end{prop}
\begin{proof}
Let $\delta = \delta_0(a_1,a_2;p)$ be as in Equation (\ref{eq:delta global min}).
Take the constant $c$ to be
$$c = p a_1 ^{-\delta} + (1-p) a_2^{-\delta}.$$
According to Calculation \ref{calc:calculus} this constant satisfies $0 < c < 1$.

Consider the   partition of the topological space $Z$ into a pair of   sublevel and superlevel sets
 $$Z = Z_{< \rho_0} \amalg Z_{\ge \rho_0}$$
where 
$$ Z_{<\rho_0 } = \{z \in Z \: : \: f(z) < \rho_0\} \quad \text{and} \quad Z_{\ge \rho_0 } = \{z \in Z \: : \: f(z) \ge \rho_0\}.$$
The subsets  $Z_{<\rho_0}$ and $Z_{\ge \rho_0}$ are respectively  open and closed in the space $Z$.

On the one hand, take the constant $ b = (a_2 \rho_0)^{-\delta}$ and observe that  every point $z \in Z_{\ge \rho_0}$ and  $\mu$-almost every element $g \in G$ satisfy
$$f(gz) \ge a_2 f(z)   \quad \Rightarrow \quad f^{-\delta}(gz) \le (a_2  f(z) )^{-\delta} \le (a_2 \rho_0)^{-\delta} = b.$$  

On the other hand, observe that  every point $ z \in Z_{<\rho_0 }$ satisfies
\begin{align*}
 A_\mu f^{-\delta}(z) &= \int_G f^{-\delta}(gz) \; \mathrm{d} \mu(g) = \\
 &= \int_{L_z} f^{-\delta}(gz) \; \mathrm{d} \mu(g) + \int_{G\setminus L_z} f^{-\delta}(gz) \; \mathrm{d} \mu(g) \le   \\
 &\le p a_1^{-\delta} f^{-\delta}(z) + (1-p) a_2^{-\delta} f^{-\delta}(z) = c f^{-\delta}(z).
\end{align*}

We conclude that the two inequalities 
$ A_\mu f^{-\delta} \le  b $ and $ A_\mu f^{-\delta} \le  c f^{-\delta}  $  hold true   on the two   subsets $Z_{\ge \rho_0}$ and $Z_{< \rho_0}$ respectively.
 Since $f^{-\delta} \ge 0$ it follows that the inequality given in Equation (\ref{eq:conclusion on contraction}) holds at  every point $z \in Z$.
\end{proof}

\subsection*{The limiting behavior of $\delta$}

We estimate  the asymptotic behavior of  the function $\delta_0(a_1,a_2;p)$ within a certain regime of the parameters $a_1, a_2$ and $p$ that   will be used below, see e.g.  Proposition  \ref{prop:existence of expanding element} and Probability \ref{prop:prob integral expands}. This is  needed in order to compute the explicit value of the constant $\delta$ appearing in our main results.

\begin{calc}
	\label{calc:existence of good delta_0}
Let   $h, \alpha, \zeta, a_0 > 0$ as well as $a_1 > 1$ be  some fixed constants. For all $\lambda > 1$ denote
	\begin{equation}
	\label{eq:a1 a2 p}
	a_{2,\lambda} = a_0 \lambda^{-h}  \quad \text{and} \quad  p_\lambda = 1-\zeta \lambda^{-\alpha}.
	\end{equation}
Then for all sufficiently large $\lambda >1 $ we have that Equation (\ref{eq:balance}) holds and the critical point
	\begin{equation}
	\label{eq:delta_lambda_def}
	\delta_\lambda = \delta_0(a_1,a_{2, \lambda};p_\lambda)
	\end{equation}
	  determined by Equation (\ref{eq:delta global min}) satisfies $\delta_\lambda > 0$.
	Moreover
	\begin{equation}
	\label{eq:explicit formula for delta as a function of lambda}
	\lim_{\lambda \to +\infty} \delta_\lambda = \frac{\alpha}{h}.
	\end{equation}
\end{calc}

\begin{proof} 
	Recall that Equation  (\ref{eq:balance})  requires   
	\begin{equation}
	\label{eq:numbers inequality}
	(1-p_\lambda) \ln \frac{1}{a_{2,\lambda}} < p_\lambda \ln a_1.
	\end{equation}
	Consider  the behavior of Equation (\ref{eq:numbers inequality}) in the limit as $\lambda \to +\infty$. Note that  
	$$\lim_{\lambda \to +\infty} p_\lambda \ln a_1 = \ln a_1.$$
	On the other hand
	$$ \lim_{\lambda \to \infty} (1-p_\lambda) \ln(1/a_{2,\lambda}) = 
	\lim_{\lambda \to \infty} \zeta \lambda^{-\alpha} \ln(a_0 \lambda^{h} )  = 0.$$
	Therefore the inequality given in  Equation (\ref{eq:numbers inequality}) is satisfied and $a_{2,\lambda} < 1$ for all $\lambda$ sufficiently large. As   shown in Calculation \ref{calc:calculus}, it follows  that  
	\begin{equation}
	\label{eq:delta_0 good}
	\delta_\lambda = - \ln_{ \left(\frac{a_1}{a_{2,\lambda}}\right) } 
	\left( - \frac{1-p_{\lambda}}{p_{\lambda}} \frac{\ln a_{2,\lambda}}{\ln a_1}\right) > 0
	\end{equation}
	for all $\lambda$ sufficiently large.	
	Lastly, the following computation 
	\begin{equation}
	\lim_{\lambda \to \infty} \delta_\lambda =\lim_{\lambda \to \infty}  -\frac{\ln \left(- \frac{1-p_{\lambda}}{p_{\lambda}} \frac{\ln a_{2,\lambda}}{\ln a_1}\right) }
	{\ln (\frac{a_1}{a_{2,\lambda}})} = 
	\lim_{\lambda \to \infty} - \frac{\ln\left(- \frac{ \zeta \lambda^{-\alpha}  \ln (a_0 \lambda^{-h}) }{\ln a_1} \right)}{\ln\left( a_1a_0^{-1} \lambda^{h}  \right)} =
	\frac{\alpha}{h}
	\end{equation}
	establishes Equation (\ref{eq:explicit formula for delta as a function of lambda}).
\end{proof}

\section{The order of an analytic function on a manifold}
\label{sec:order}

Let $k$ be a local field. Consider a    $k$-analytic\footnote{ We refer the reader to \cite{bourbaki2007varietes} and \cite[Part II]{serre2009lie} for basic information on $k$-analytic functions and manifolds.}    manifold $M$ of dimension $d \in \NN$.  
Let   $\mathcal{A}(M,k^m)$ denote the $k$-vector space of all $k$-analytic maps $f : M \to k^m$ for each dimension $m\in\NN$.

The \emph{order} $\mathrm{ord}_p f$ of any given $k$-analytic function $f \in \mathcal{A}(M,k)$ at the point $p \in M$ is defined as follows. Fix an arbitrary local chart  $\varphi : U \to V$ where $p \in U \subset M$ and $0 \in V \subset k^d$ are open subsets such that $\varphi(p) = 0$. The $k$-analytic function $f$ can be written  in local coordinates as
$$ (f \circ \varphi^{-1})(x) = \sum_\beta c_\beta x^\beta \quad \forall x \in V$$
where  the $\beta$'s are multi-indices and   the coefficients $c_\beta$ belong to the local field $k$. Then
$$\mathrm{ord}_p f = \inf_\beta \{ |\beta| \: : \: c_\beta \neq 0 \}. $$
If the function $f$ is identically zero on a neighborhood of the point $p$ then all the coefficients $c_\beta$ are zero and  we set $\mathrm{ord}_p f = \infty$. The fact that transition maps are $k$-analytic isomorphisms implies that the definition of $\mathrm{ord}_p f$ is independent of the local chart.

The \emph{order} of any $k$-analytic function $ f = (f_1,\ldots,f_m) \in \mathcal{A}(M,k^m)$ is   
\begin{equation}
\mathrm{ord}_p(f) = \inf_{i \in \{1,\ldots,m\} } \mathrm{ord}_p (f_i).
\end{equation}
Let $\mathcal{F} \subset \mathcal{A}(M,k^m)$ be a family of $k$-analytic functions. The \emph{order of the family} $\mathcal{F}$ at the point $ p\in M$ is  
\begin{equation} \mathrm{ord}_p \mathcal{F} = \sup_{f \in \mathcal{F}} \mathrm{ord}_p f.\end{equation}
Moreover denote
\begin{equation}  \mathrm{ord}_M \mathcal{F} = \sup_{p \in M} \mathrm{ord}_p \mathcal{F}.\end{equation}

\begin{remark*} The order of the product $fg$ of any pair of $k$-analytic functions $f, g \in \mathcal{A}(M,k)$ satisfies 
\begin{equation}
\label{eq:order product}
\mathrm{ord}_p(fg) = \mathrm{ord}_p(f) + \mathrm{ord}_p(g) \quad \forall p \in M.
\end{equation}
\end{remark*}

\begin{lemma}
\label{lem:order on embedded submanifolds}
Let $N$ be an embedded $k$-analytic submanifold of $M$. Then every $k$-analytic function $ f\in\mathcal{A}(M,k^m)$ satisfies
$ \mathrm{ord}_p f \le \mathrm{ord}_p {f_{|N}}$ for all points $ p \in N$.
\end{lemma}
\begin{proof}
Denote $e = \dim_k N$. Consider an arbitrary point $ p \in N$. By   assumption there is an relative neighborhood $p \in U_1 \subset N$ and a  local chart $\varphi_1 : U_1 \xrightarrow{\cong} V_1 \subset k^{e}$ with $\varphi_1(p) = 0$ such that the power series expressing the function $f \circ \varphi_1^{-1}$ at the point $0$ admits a non-zero monomial of degree  $\mathrm{ord}_p f_{|N}$. The $k$-vector space $k^d$ admits a direct sum decomposition $k^d = k^e \oplus k^{d-e}$ such that, up to replacing $U_1$ and $V_1$ by  suitable smaller    neighborhoods, there is an open neighborhood $p \in  U \subset M$  and a local chart $\varphi : U \to k^d$ with $\varphi(p) = 0$, $\varphi(U) = V_1 \times V_2$ for some  subset $V_2 \subset k^{d-e}$   so that $U_1 \subset U$ and $\varphi_{|U_1} = \varphi_1$  \cite[Part II, Chapter III, \S10.1]{serre2009lie}. It is clear that in this situation $\mathrm{ord}_p f \le \mathrm{ord}_p f_{|N}$. As the point $p\in N$ was arbitrary the conclusion follows.
\end{proof}

\subsection*{Strong order}

In working with matrix coefficients of unipotent groups we will find it convenient to use the following sharper notion.

\begin{definition}
\label{def:strong order}
The $k$-analytic function $f \in \mathcal{A}(M,k^m)$ has \emph{strong order} at most $d \in \NN$ at the point $p\in M$ is there exists an embedded $k$-analytic submanifold  $N$ admitting a local chart $\varphi : U \to V$ where $p \in U \subset N$ and $V \subset k^{\dim_k N}$ are open subsets and such that $ f \circ \varphi^{-1 } : V \to k^m $ is a non-zero \emph{polynomial} mapping of degree at most $d$.  
\end{definition}

 Lemma \ref{lem:order on embedded submanifolds} implies that strong order gives an upper bound on order. Namely the following is true.
 
\begin{lemma}
\label{lem:strong order bounds order}
If $f \in \mathcal{A}(M,k^m)$ has strong order at most $d \in \NN$ at the point $ p \in M$ then $\mathrm{ord}_p f \le d$.
\end{lemma}
%
%
%


\subsection*{Order of matrix coefficients of compact Lie groups}

 Let $V$ be a real   finite dimensional    inner product space. Let $K$ be a  compact connected subgroup of the orthogonal group $\mathrm{O}(V)$. In particular $K$ is a real Lie group. Let $T$ be a maximal compact torus of the group $K$ with real Lie algebra $\mathfrak{t} = \mathrm{Lie}(T)$.

\begin{lemma}
\label{lem:order on torus}
Let $\Omega  \subset i\mathfrak{t}^*$ be a finite set of analytically integral forms. Let $c_\omega \in \CC$ be an arbitrary coefficient for each $\omega \in \Omega$. Consider the real analytic function
$ \zeta \in \mathcal{A}(T,\CC)$ given by
\begin{equation}
\zeta(\exp(X)) = \sum_{\omega\in\Omega} c_\omega e^{\omega(X)} \quad \forall X \in \mathfrak{t}.
 \end{equation}
If the function $\zeta$ is not identically zero on the torus $T$ then 
$ \mathrm{ord}_T \zeta < |\Omega|$.
\end{lemma}

The  function $\zeta$ is defined on the real Lie group $T$ and  is regarded as a real analytic   mapping into the two-dimensional real vector space $\CC  \cong \RR^2$.
\begin{proof}[Proof of Lemma \ref{lem:order on torus}]
Fix an arbitrary element $t_0 \in T$  with $t_0 = \exp(X_0)$ for some  $X_0 \in \mathfrak{t}$. We will compute  the order of the function $\zeta$ at the point $t_0$ and show that $\mathrm{ord}_{t_0} \zeta < |\Omega|$.

Choose    neighborhoods $t_0 \in U \subset T$ and   $0 \in V \subset \mathfrak{t}$ such that the mapping $\varphi : U \xrightarrow{\cong} V$ given by $\varphi : \exp (X_0 + X ) \mapsto X$ for all $X \in V$ is  a local chart. For a given element $t   \in U$ with $\varphi(t) = X \in V$ write
$$ \zeta(t) = \zeta(\exp(X_0 + X))=  
 \sum_{\omega\in\Omega} c_\omega e^{\omega(X_0)} \sum_{l=0}^\infty  \frac{ \omega(X)^l }{l!} =
  \sum_{l=0}^\infty \frac{1}{l!} \sum_{\omega \in\Omega} \bar{c}_\omega  \omega( X)^{l} $$
where $\bar{c}_\omega = c_\omega e^{\omega(X_0)}$. The part $\zeta_0$  of the real analytic function $\zeta$ consisting of all monomials with multi-index $\beta$ satisfying $|\beta| < |\Omega|$ equals
$$ \zeta_{0}(t) = \sum_{l=0}^{|\Omega|-1} \frac{1}{l!} \sum_{\omega \in \Omega} \bar{c}_\omega  \omega(  X)^{l}$$
for each given element $t = \exp(X_0 + X) \in U$ as above.

The subspace $\ker (\omega_{1}-\omega_{2})$ of the real Lie algebra  $\mathfrak{t}$ is proper for each pair of distinct forms $\omega_1,\omega_2 \in \Omega$.
Therefore  there is an element $Y \in V$ such that $\omega_{1}( Y ) \neq \omega_{2}(Y)$ for all such pairs of distinct indices. 
The   Vandermonde   determinant formula gives
$$ \det\left( \frac{\omega(Y)^l}{l!} \right)_{\substack{\omega \in \Omega \\ l \in \{0,\ldots,|\Omega|-1\}}} = 
\prod_{l=0}^{|\Omega|-1} \frac{1}{l!} \cdot \prod_{\substack{\omega_1, \omega_2 \in \Omega \\ \omega_1 \neq \omega_2} }  \left(\omega_{1}(Y) - \omega_{2}(Y) \right) \neq 0.$$
The non-vanishing of the above determinant implies   $\zeta_{0}(Y) \neq 0$ so that $\mathrm{ord}_{t_0} \zeta < |\Omega|$ as required. 
As the point $t_0 \in T$ in the above argument was arbitrary we   conclude that   $\mathrm{ord}_T f < |\Omega|$.
\end{proof}

Consider the Hermitian space $V_\CC = V \otimes \CC$ which is the complexification of the inner product space $V$. The maximal compact torus $T$ preserves the Hermitian form on $V_\CC$.   Let $\Omega = \Omega(T,V)$ be the weights  of the torus $T$ in its representation on the Hermitian space $V_\CC$. Every weight $\omega \in \Omega $ is an analytic integral form belonging to $i \mathfrak{t}^*$. Let 
$$V_\CC = \bigoplus_{\omega \in \Omega} W_\omega$$
 be the associated weight space decomposition satisfying
\begin{equation}
 \exp(X) w = e^{\omega(X) } w \quad \forall  \omega \in \Omega, w \in W_\omega,X \in \mathfrak{t}.
 \end{equation}


\begin{lemma}
\label{lem:matrix coef for compact}
Consider the matrix coefficient $\zeta_{u,v} \in \mathcal{A}(K,\RR)$ given by 
$$ \zeta_{u,v} : K \to V, \quad \zeta_{u,v}(g) = \left<gu,v\right> \quad \forall g \in K. $$
for some pair of vectors $u,v  \in V$. If $\zeta_{u,v}$ is not identically zero on the group $K$ then  $\mathrm{ord}_K \zeta_{u,v} < |\Omega(T,V)| \le \mathrm{dim}_\RR V$.
\end{lemma}
\begin{proof}
Every element of the compact connected  Lie group $K$ belongs to some maximal torus  and all such  tori are conjugate \cite[Corollary 4.35, Theorem 4.36]{knapp2013lie}. Up to replacing the vectors $u$ and $v$ by   $gu$ and $gv$ for some element $g \in K$, we may assume without loss of generality that   the restriction of the matrix coefficient $\zeta_{u,v}$ to the   maximal torus $T$ is not identically zero.


Denote $\Omega = \Omega(T,V)$ and consider the weight space decompositions  $u = \bigoplus_{\omega \in \Omega} u_\omega$ and $v = \bigoplus_{\omega\in\Omega} v_\omega$ where $u_\omega, v_\omega \in W_\omega$. 
The   matrix coefficient $\zeta_{u,v}$ can be expressed  at the point $g = \exp X \in T $ with   $ X \in \mathfrak{t}$ as
\begin{equation}
\zeta_{u,v}(g) =  \left<\exp(X) u,v\right> =
 \left< \exp(X) \sum_{\omega \in \Omega}  u_\omega, \sum_{\omega\in \Omega} v_\omega \right> = 
 \sum_{\omega\in\Omega} \left<u_\omega,w_\omega\right>   e^{\omega(X)}. 
\end{equation}
It follows from   Lemmas \ref{lem:order on embedded submanifolds} and  \ref{lem:order on torus} applied at the identity element that
$$  \mathrm{ord}_ {e_K} \zeta_{u,v} \le  \mathrm{ord}_ {e_T} \zeta_{u,v|T} < |\Omega|. $$

Observe that $\zeta_{u,v}(gh) = \zeta_{hu,v}(g)$ for all pairs of elements $g,h \in K$. By repeating the above argument with respect to   each matrix coefficient of the form $\zeta_{hu,v}$ we show that $\mathrm{ord}_K \zeta_{u,v} < |\Omega|$  as required.
\end{proof}

Fix some  $l \in \NN$. Consider the $l$-exterior power inner product space $\bigwedge^l V$.  The compact connected group $K$ can naturally be regarded as a subgroup of $\mathrm{O}(\bigwedge^l V)$. The weights $ \Omega(T, \bigwedge ^l V_\CC)$ of the torus $T$ in its representation on the Hermitian space $\bigwedge^l V_\CC$ satisfy
$$ \Omega(T, \bigwedge ^l V_\CC) \subset \Omega^l =  \{ \omega_1 + \cdots + \omega_l \: : \: \omega_1,\ldots,\omega_l \in \Omega\}.$$
 The following     follows immediately from the previous Lemma \ref{lem:matrix coef for compact}.

\begin{lemma}
\label{lem:order for wedge product}
Consider the matrix coefficient $\zeta^l_{u,v} \in \mathcal{A}(K,\RR)$ given by 
$$ \zeta_{u,v}^l(g) = \left<\bigwedge^l g u,v  \right> \quad \forall g \in K.$$
for some pair of vectors $u,v  \in \bigwedge^l V$. If   $\zeta^l_{u,v}$ is not identically zero on the group $K$ then  
$ \mathrm{ord}_K \zeta_{u,v}^l < |  \Omega(T, \bigwedge ^l V_\CC)| \le |\Omega^l|$.
\end{lemma}

\begin{remark*}
The above proofs of   Lemmas \ref{lem:matrix coef for compact} and  \ref{lem:order for wedge product} apply more generally without assuming the compact group $K$   is connected at each connected component where   the matrix coeffcient in question is non-zero.
\end{remark*}
%

\subsection*{Strong order of unipotent groups in positive characteristic} Assume that $k$ is a non-Archimedean local field with $ p=  \mathrm{char}(k)  > 0$. Let $V$ be a finite dimensional $k$-vector space of dimension $ d = \dim_k V$.

Consider an arbitrary element $s \in \mathrm{GL}(V)$. Write $n = s - \mathrm{Id} \in \mathrm{End}(V)$. Note that $s^p = \mathrm{Id}$ if and only if    $n^p = 0$. If this is the case then  the Cayley--Hamilton theorem implies that   $n^{d} = 0$.  In particular
\begin{equation}
\label{eq:log power series}
 \log(s) = \log(\mathrm{Id} + n) = \sum_{i=1}^{d-1} (-1)^{i+1} \frac{n^i}{i}.
\end{equation}
The matrix $x = \log(s)$ is a linear combination of positive powers of the matrix $n$ and as such $x^d = 0$. Therefore
\begin{equation}
s = \exp(x) = \exp(\log s) = \sum_{i=0}^{d-1} \frac{x^i}{i!}.
\end{equation}

\begin{lemma}
\label{lem:orbit map for unitpotent subgroup}
Let $T$ be a diagonal $k$-split torus subgroup of $\mathrm{GL}(V)$.  Let $\theta : k^r \to \mathrm{GL}(V)$ be a $k$-rational representation for some $r \in \NN$ so that $U = \theta(k^r) \le \mathrm{GL}(V)$ is a connected $k$-unipotent subgroup.   Assume that $T$ normalizes $U$ and that there is a $k$-character $\alpha \in X(T)$ so that 
\begin{equation}
\label{eq:conjugation equation}
 t \theta(z) t^{-1} = \theta(\alpha(t) z) \quad \forall t \in T, z \in k^r
 \end{equation}
Then the inclusion mapping $U \hookrightarrow \mathrm{End}(V)$ has strong order at most  $2dr I $  at all points,  where $I \in \NN$ is the  maximal number so that $I \alpha = \beta  - \beta'$ for any pair of distinct $k$-characters $\beta,\beta' \in X(T)$ corresponding to  diagonal entries of $T$.
\end{lemma}



\begin{proof}
Assume to begin with that $r = 1$. In other words $\theta$ is a $k$-rational representation of the additive group. The explicit form of such  representations is known  \cite[II.\S2.2.6]{demazure1980introduction}. Namely there is some   $N \in \NN \cup \{0\} $ and a family of pairwise commuting matrices $s_0,\ldots,s_N \in \mathrm{GL} (V)$ satisfying $s_i^p = \mathrm{Id}$ for all $ i \in \{0,\ldots,N\}$ such that
\begin{equation}
\label{eq:additive rep}
 \theta(z) = \prod_{i=0}^N \exp( z^{p^i} \log s_i ) \quad \forall z \in k.
 \end{equation}

Denote $x_i = \log s_i$. The expression  $t \theta(z) t^{-1}$  for each element $t \in T$ and $z \in k$ can be computed in two different ways. On the one hand Equation (\ref{eq:conjugation equation}) gives
 \begin{equation}
 \label{eq:log 1}
t \theta(z) t^{-1}=   \theta(\alpha(t) z)  = \prod_{i=0}^N \exp( (\alpha(t) z)^{p^i} x_i  ) = \prod_{i=0}^N \exp(  z ^{p^i} \alpha(t) ^{p^i} x_i  ) \quad \forall z \in k.
  \end{equation}
 On the other hand  
 \begin{equation}
\label{eq:log 2}
t \theta(z) t^{-1}=      \prod_{i=0}^N t \exp(  z^{p^i} x_i  ) t^{-1} = \prod_{i=0}^N \exp(  z ^{p^i} t    x_i t^{-1} ) \quad \forall z \in k.
 \end{equation}
We apply logarithm to both Equations (\ref{eq:log 1}) and (\ref{eq:log 2}). Since the matrices $x _i  $ pairwise commute we obtain the equality
\begin{equation}
\label{eq:log 3}
 \sum_{i=0}^N z^{p^i} (\alpha(t) ^{p^i} x_i  ) = \sum_{i=0}^N z^{p^i} ( t x_i  t ^{-1}) \quad \forall z \in k, t \in T.
  \end{equation}
Note that Equation (\ref{eq:log 3}) holds true if and only if 
$ \alpha(t) ^{p^i} x_i =  t x_i  t^{-1} $ for all $i \in \{0,\ldots,N\}$.  This only be the case for each given index $i$ provided there are two distinct $k$-characters $\beta, \beta' \in X(T)$ corresponding to diagonal entries of the torus $T$   such that $p^i \alpha = \beta - \beta'$.  Therefore the number $N$ must  be such that $p^N \le I$ where $I$ is as in the statement.

The discussion preceding the Lemma shows that each function $ \exp(z^{p^i} \log s_i )  $ is a polynomial map in the variable $z$ into the $k$-vector space $\mathrm{End}(V)$ of degree at most $p^i \dim_k V$ for every $i \in \{0,\ldots,N\}$.   Therefore $\theta$ is a polynomial map in the variable $z$ into the $k$-vector space $\mathrm{End}(V)$ of degree at most 
$$   \sum_{i=0}^N p^i  \dim_k V \le  \frac{p^{N+1} -1}{p-1}\dim_k V \le  2 I \dim_k V. $$

Lastly if $ r > 1$ then $\theta$ is obtained as a product in the endomorphism ring $\mathrm{End}(V)$ of $r$-many $k$-rational representations as above. The result follows. 
 \end{proof}

%

\section{Sublevel set estimates  of analytic functions on manifolds}
\label{sec:good functions}

Let $k$ be either the field $\RR$ or a non-Archimedean local  field\footnote{Our discussion and results easily extend to the complex case via  a restriction  of scalars. This has the effect of doubling the dimension of the domain in Theorem \ref{thm:analytic are good} and its applications below.} 
and $|\cdot|$ be an absolute value on $k$. Fix a dimension $ m\in\NN$.  Endow  the   $k$-vector space $k^m$  with the supremum norm $\|\cdot\|_\infty$   given by 
\begin{equation}
\|(x_1,\ldots,x_m)\|_\infty = \sup_{i} |x_i| \quad \forall x_1,\ldots,x_m \in k.
\end{equation}

Let $M$ be a fixed compact $k$-analytic manifold. Let $\eta$ be a probability measure in the canonical  measure class of   $M$. For example, if $M$ is a compact $k$-analytic group then    $\eta$ can be taken to be    normalized Haar 
measure.

Recall from \S\ref{sec:order} that  $\mathcal{A}(M,k^m)$ denotes the $k$-vector space of all $k$-analytic maps $f : M \to k^m$.
Every non-empty subset $X \subset M$ defines a   seminorm  $\|\cdot\|_X$ on $\mathcal{A}(M,k^m)$    by 
\begin{equation}
\|f\|_X=\sup_{x\in X}  \|f(x)\|_\infty \quad \forall f \in \mathcal{A}(M,k^m).
\end{equation}
Endow   $\mathcal{A}(M,k^m)$ with the topology coming from the norm $\|\cdot\|_M$.






The goal of the  current \S\ref{sec:good functions} is to prove Theorem \ref{thm:goodness on a compact group}. This  establishes
an upper bound on the $\eta$-measure of sublevel sets of   $k$-analytic maps on the compact $k$-analytic manifold $M$.

\begin{thm}
	\label{thm:goodness on a compact group}
	Let $\mathcal{F} \subset \mathcal{A}(M,k^m)$ be a compact  family of $k$-analytic maps. If 
 $\dim_k \mathrm{span}_k \mathcal{F} < \infty$ then there are constants $\kappa_\mathcal{F},\varepsilon_\mathcal{F} > 0$   such that every   map $f \in \mathcal{F}$ satisfies
	\begin{equation}
	\label{eq:good compact set}
	\eta( \{x \in M \: : \: \|f(x)\|_\infty < \varepsilon \}) \le \kappa_\mathcal{F}  \varepsilon ^{\frac{1}{\dim_k M \mathrm{ord}_M \mathcal{F}}} \quad \forall 0 < \varepsilon < \varepsilon_\mathcal{F}.
	\end{equation}
\end{thm}


Note that the dimension $m$  of the target space  does not appear in Equation (\ref{eq:good compact set}). 
 
 
We deduce Theorem \ref{thm:goodness on a compact group} from known results on $(C,\alpha)$-good functions that are  discussed below. In fact  Theorem \ref{thm:goodness on a compact group} can be seen as an analogue of Theorem \ref{thm:analytic are good} in the setting of $k$-analytic manifolds.


\subsection*{Sublevel set estimates} 
Let $d \in \NN$ be fixed. A \emph{ball} in the $d$-dimensional $k$-vector space  $k^d$ with center point  $y   \in k^d$ and   radius $r > 0$ is  the subset
$$B = B(y,r) = \{x   \in k^d \: : \: \|x-y\|_\infty \le r \}.$$ 
Let $V \subset k^d$ be a fixed open subset.
Given a  continuous function $ f : V \to k $ and a   ball  $B \subset V$ denote
\begin{equation}
\label{eq:max norm}
\|f\|_B = \sup_{x \in B} |f(x)|.\end{equation}
In addition, for every $\varepsilon > 0$ denote
\begin{equation} B^{f,\varepsilon} = \{x \in B \: : \: |f(x)| < \varepsilon \}.\end{equation}

Let $\lambda$ be a Haar measure on the $k$-vector space $k^d$ regarded as an additive  locally compact group. 
The following terminology has been introduced in \cite{KM}. 

\begin{definition}
\label{def:good}
A continuous function $f :V \to k$ is \emph{$(C,\alpha)$-good}   if   every open ball $B \subset V$     satisfies
\begin{equation}
\label{eq:good def}
 \lambda(B^{f,\varepsilon}) \le C \left(\frac{\varepsilon}{\|f\|_B}\right)^\alpha  \lambda(B)
\end{equation}
 for every $\varepsilon > 0$.
 \end{definition}

Polynomials are $(C,\alpha)$-good   with parameters $C$ and $\alpha$ depending only on  degree. In the real case this goes back to the Remez inequality \cite{remez1936propriete, brudnyui1973extremal}. See also \cite[Lemma 4.1]{dani1993limit}, \cite[Proposition 3.2]{KM} and \cite[Proposition 2.8]{aka2015diophantine}.

Analytic functions are also known to be $(C,\alpha)$-good under suitable assumptions:




\begin{theorem}
\label{thm:analytic are good}
Let $f : V \to k$ be a $k$-analytic function. Assume that there exist constants $A_1, A_2 > 0$ and some $l \in \NN$ such that
$$ \|\partial_\beta f\|_V \le A_1 \quad\quad \text{$\forall$multi-index $\beta$ with $ |\beta| \le l$} $$
and 
$$ |\partial_i^l f(x)| \ge A_2 \quad  \quad \forall x \in V, \; \forall  i = 1,\ldots, d.$$
Then there is a constant $C > 0$  such that $f$ is $(C, \frac{1}{dl})$-good.
The constant $C$ depends only on the values of $A_1,A_2,l$ and on the dimension $d$.
\end{theorem}
\begin{proof}
The real case
 is precisely  \cite[Lemma 3.3]{KM}. Note that our ``balls" are ``cubes" in the terminology of \cite{KM}.  The non-Archimedean case is contained in \cite[Theorem 3.2]{kleinbock2005flows}. 
%
\end{proof}

The following argument is well-known. In particular, it is used implicitly in the proof of \cite[Proposition 3.4, p. 349]{KM} and    explicitly in \cite[Proposition 4.2]{kleinbock2005flows}. We provide a detailed proof for the reader's convenience.

\begin{prop}
\label{prop:may rotate partial derivatives}
Let $f : V \to k$ be a $k$-analytic function. Assume that there is a constant $A'_1 > 0$ and some $l \in \NN$ such that
$$ \|\partial_\beta f\|_V \le A'_1 \quad \text{$\forall$multi-index $\beta$ with $ |\beta| \le l$}. $$
Assume moreover that $0 \in V$ and  that there is a multi-index $\beta_0$ with $|\beta_0| = l$ as well as a constant $A'_2 > 0$ satisfying
$$|\partial_{\beta_0} f(0)| \ge A'_2. $$ 
Then there constants $A_1, A_2 > 0$ and $ K \ge 1$ as well as a    linear map $g \in \mathrm{GL}(k^d)$ with $1 \le | \det g | \le K$ such that 
 $$ \|\partial_\beta (f \circ g)\|_V \le A_1 \quad \text{and} \quad  |\partial_i^l (f \circ g)  (0) | \ge A_2 \quad \forall i \in \{1,\ldots,d\}.  $$
The     constants $A_1, A_2$ and $K$  depend only on the two constants $A'_1$ and $A'_2$ and on the local field $k$.
\end{prop}

\begin{proof} 
Let $\beta$ be  an arbitrary   multi-index  with $|\beta | = l$. Consider the value of the partial derivative $\partial_{\beta}(f \circ g)(0)$ where the element $g$ varies over the ring of endomorphisms $\mathrm{End}(k^d)$.
 Identify the space $\mathrm{End}(k^d)$ with the $k$-affine space of matrices $\mathrm{M}_d(k)$.   A repeated application of the chain rule shows that the value of the partial derivative $\partial_{\beta}(f \circ g)(0)$  is given by a homogeneous polynomial $P_{\beta,f}$ in the matrix entries of      $g$. The   coefficients  of the polynomial $P_{\beta,f}$ are the partial derivatives  $\partial_{\beta'} f(0)$ where $\beta'$ varies over \emph{all}   possible multi-indices with $|\beta'| = l$. In particular, the  polynomial $P_{\beta,f}$ is non-trivial as one of its coefficients is given by $\partial_{\beta_0} f(0)$.

The above  discussion applied with respect to the  multi-indices  corresponding to the partial differentiation operators $\partial_i^l$   implies that there exists a non-empty Zariski open subset of  $  \mathrm{End}(k^d)$ consisting of endomorphisms $g$ satisfying 
\begin{equation}
\label{eq:Zariski open condition}
\partial_i^l (f \circ g)  (0) \neq 0 \quad \forall i \in \{1,\ldots,d\}.
\end{equation}
Therefore we may  find a  linear map $g$ belonging to the Zariski open subset   $\mathrm{GL}(k^d)$ and satisfying Equation (\ref{eq:Zariski open condition}).
 As the polynomials $\mathrm{P}_{\beta,f}$ are all homogeneous we may assume, up to replacing the matrix $g$ by a scalar multiple,    that   $1 \le | \det g | = K$ for some constant\footnote{In the Archimedean case   $K=1$. Namely we may assume that $|\det g| = 1$.} $K \ge 1$ depending only on the local field $k$. 

The coefficients of the polynomials in the family  $P_{\beta,f}$   satisfy $|\partial_{\beta'}f(0)|  \le A'_1$ and $|\partial_{\beta_0}f(0)|  \ge A'_2$ for some multi-index $\beta_0$. The family of all   polynomials satisfying such coefficient bounds is compact. This implies the existence of the   constants $A_1$ and $A_2$ depending only on the pair of constants $A'_1$ and $A'_2$ and the local field $k$.
\end{proof}



We are now ready to complete the proof of the main result for  \S\ref{sec:good functions}.

\begin{proof}[Proof of Theorem \ref{thm:goodness on a compact group}]


Fix an arbitrary point  $x \in M$. Consider a compatible chart    given by an open subset $  U_x  \subset M$ containing the point $x$ and a map 
$\varphi_x : U_x \to V_x$ with $\varphi_x(x) = 0$   where $V_x \subset k^d$ is a ball at the point $0$.  Let $\psi_x : V_x \to U_x$ denote the  inverse mapping of $\varphi_x $. 




For every multi-index $\beta$ consider the two seminorms $\|\cdot\|_{x,\beta}$ and $\|\cdot\|_{x,\beta}^0$ on the $k$-vector space $\mathcal{A}(M,k^m)$ given by
$$\|f\|_{x,\beta} = \|\partial_\beta(f \circ \psi_x) \|_{V_x} \quad \text{and} \quad \|f \|_{ x,\beta}^0 = \|\partial_\beta(f \circ \psi_x)(0) \|_\infty \quad \forall f \in \mathcal{A}(M,k^m). $$

Our assumptions that the family $\mathcal{F}$ is compact and that  $\dim \mathrm{span}_k(\mathcal{F}) < \infty$   imply that there are constants $A'_{1,x}, A'_{2,x} > 0$ such that 
$$ \|f\|_{x,\beta}    \le A'_{1,x} \quad \forall f \in \mathcal{F} \;\; \forall \text{multi-index $\beta$ with $0 \le |\beta| \le \mathrm{ord}_x \mathcal{F} + 1$} $$
and
$$ \|f \|_{x, \beta_0}^0 \ge A'_{2,x} \quad \forall f \in \mathcal{F}  \; \; \text{and for \emph{some} multi-index $\beta_0 = \beta_0(f)$ with $0 \le |\beta_0| \le \mathrm{ord}_x \mathcal{F}$}. $$

According to Proposition \ref{prop:may rotate partial derivatives} there are constants $A_{1,x}, A''_{2,x} > 0$ such that for every function $f \in \mathcal{F}$  there exists a linear isomorphism $g = g(x,f) \in \mathrm{GL}(k^d)$ with $\|f \circ g\|_{x,\beta} \le A_1$ and 
$$ \|f \circ g \|_{x, \partial_i^l}^0  \ge A''_{2,x} \;\; \text{for some $l =l(f) \in \{0,\ldots, \mathrm{ord}_x f\}$ and     $ \forall i \in \{1,\ldots,d\}$}.$$
We may assume that the collection of elements $g(x,f)$ for $f \in \mathcal{F}$ is relatively compact.

 Finally,   there is a constant $A_{2,x} > 0$ such that, up to replacing $U_x$ by a   smaller  neighborhood of the point $x \in M$ and $V_x$ by a corresponding   ball of smaller radius, 
 the function $f \circ \psi_x \circ g(x,f)$ is well-defined on the ball $V_x$ and satisfies
 $$ \inf_{y \in V_x} \| \partial_i^l (f \circ \psi_x \circ g(x,f))(y) \|_\infty \ge A_{2,x} \quad \forall i \in \{1,\ldots,d\}$$  
for all functions $f \in \mathcal{F}$ and for some $l = l(f)\in \{0,\ldots,\mathrm{ord}_x \mathcal{F} \}$. In the real case this  follows from the error estimate in the  multi-dimensional variant of   Taylor's series. In the  non-Archimedean case this follows from the non-Archimedean triangle inequality.

The map $\psi_x$ pushes forward the measure class of the Haar measure $\lambda$ on the ball $V_x $ to the measure class of $\eta$ on $V $. Let $0 < D_x < \infty$ be the supremum of the  $L_\infty$-norm of the   Radon--Nikodym derivative $\frac{\mathrm{d} (\psi_x \circ g) _{*} \lambda}{\mathrm{d}m}$ as $g$ ranges over all possible elements $g = g(f,x)$.

We know from Theorem \ref{thm:analytic are good} that each map $f \circ \varphi_x \circ g : V_x \to U_x$ has at least one  coordinate which is $(C_x, \frac{1}{dl})$-good on $V_x$   where $C_x > 0$ is some constant depending only on the constants $A_{1,x}$ and $A_{2,x}$. This means that 
\begin{equation}
\label{eq:applying goodness locally}
\eta(V_x^{f,\varepsilon}) \le D_x \lambda(V_x^{f \circ \psi_x \circ g,\varepsilon}) \le C_x D_x \left(\frac{\varepsilon}{\|f\circ\psi_x \circ g \|_{V_x}}\right)^{\frac{1}{dl}}  \lambda(V_x) \quad \forall f\in \mathcal{F}, \forall  \varepsilon > 0.
\end{equation}




Recall that the $k$-analytic variety $M$ is assumed to be compact. Therefore there is some $N \in \mathbb{N}$ and some points $x_1,\ldots,x_N \in M$  such that $M \subset \bigcup_{i=1}^N U_{x_i}$. Denote $D = \max_{i=1}^N D_{x_i}$ and $C = \max_{i=1}^N C_{x_i}$. 

We know that $\sup_{f \in \mathcal{F}} \|f\| < \infty$. Moreover, for all $ i \in \{1,\ldots,N\}$, we have $\inf_{f \in \mathcal{F}} \|f\|_{U_{x_i}} > 0$, as $\mathrm{ord}_x \mathcal{F} \le l$. Therefore there is some constant $B > 0$ such that    
\begin{equation}
\label{eq:ratios}
\inf_{i \in \{1,\ldots,N\}} \inf _{f\in\mathcal{F}} \frac{\|f \circ \psi_{x_i} \|_{B_{x_i}}}{\|f\|_M}  = \inf_{i \in \{1,\ldots,N\}} \inf _{f\in\mathcal{F}} \frac{\|f\|_{U_{x_i}}}{\|f\|_M} \ge B. 
\end{equation}


It remains to deduce the global Equation (\ref{eq:good compact set}) from the local information of Equation (\ref{eq:applying goodness locally}).  For every function $p \in \mathcal{F}$ and every $\varepsilon > 0$ we have that
\begin{equation}
\label{eq:everything together}
\eta(M^{f,\varepsilon}) \le
 \sum_{i=1}^N \eta(U_{x_i}^{f,\varepsilon}) \le 
 CD  \sum_{i=1}^N
\left(\frac{\varepsilon}{\|f   \|_{U_{x_i}}}\right)^{\frac{1}{dl}}
\lambda (B_{x_i}). 
\end{equation}
The proof follows  for  an appropriate choice of the constant $\kappa_\mathcal{F} > 0$ by putting together Equations (\ref{eq:applying goodness locally}),   (\ref{eq:ratios}) and (\ref{eq:everything together}).
\end{proof}

\section{Grassmannians of normed vector spaces}
\label{sec:grassmanians}


Let  $k$ be the field $\RR$  or a non-Archimedean local  field with absolute value $|\cdot|$. 
Let $V$ be a $k$-vector space of dimension $ n =  \dim_k V  $ equipped with the norm $\|\cdot\|$.
Fix a non-trivial direct sum decomposition 
$$ V = U \oplus U' $$
 and a pair of projections 
$$ \mathrm{P} : V \to U \quad \text{and} \quad \mathrm{P}' : V \to U'.$$


%
Fix some $l \in \NN$.  Let $K$ be a fixed compact $k$-analytic subgroup of $\mathrm{GL}(V)$ with the  Haar probability measure $\eta_K$. For each  vector $x \in \bigwedge^l V$   consider the $k$-analytic function $Q_{x}$ on the compact group $K$   defined by
\begin{equation}
\label{eq:Qx}
Q_{x} \in \mathcal{A}(K, \bigwedge^l U), \quad Q _{x} (g) =  \left( \bigwedge^l \mathrm{P}  g\right) x \quad \forall g \in K. 
\end{equation}


\begin{thm}
	\label{thm:goodness: probability of q being small}
  Let 
   $\mathfrak{N}$  be a   closed subset of the  Grassmannian $ \mathfrak{Gr}(l,V)$. Consider the family $\mathcal{F}(\mathfrak{N}) \subset \mathcal{A}(K, \bigwedge^l U)$ of $k$-analytic functions   given by
\begin{equation}
\label{eq:FN family} 
\mathcal{F}(\mathfrak{N}) = \{ Q_{w_1 \wedge \cdots \wedge w_l }
 \: : \: \text{$\|w_i \| \le 1$ and $\mathrm{span}_k\{w_1,\ldots,w_l\} \in  \mathfrak{N}$} \}.
 \end{equation}
Then there are constants $\kappa_\mathfrak{N},\varepsilon_\mathfrak{N} > 0 $ such that   
\begin{equation}
\label{eq:desired estimate}
 \eta_K(\{g \in K \: : \: \inf_{0 \neq w \in gW} \frac{\| \mathrm{P}w\|}{\|w\|}    \le \varepsilon \}) < \kappa_\mathfrak{N} \varepsilon^{\frac{1}{ \mathrm{dim}_k K \mathrm{ord}_K \mathcal{F}(\mathfrak{N})}} 
  \end{equation}
for every subspace $W \in \mathfrak{N}$ and   every $0 < \varepsilon < \varepsilon_\mathfrak{N}$. 
\end{thm}


In proving Theorem \ref{thm:goodness: probability of q being small} we may and will assume without loss of generality that $\mathrm{ord}_K \mathcal{F}(\mathfrak{N}) < \infty$. In other words, for  every subspace $W \in \mathfrak{N}$ and every open subset $U \subset K$ there is an element $g \in U$ with $gW \cap U' = \{0\}$. In particular we may assume that that $l \le \dim_k U$.

\subsection*{Bijections and contraction}

In   \S\ref{sec:nilpotent} below  we will be relying on  Theorem \ref{thm:goodness: probability of q being small} in combination with the following elementary observation.

\begin{lemma}
	\label{lem:bound on norm inverse in terms of kappa}
	Let $L : V \to V$ be a linear bijection that preserves the two subspaces $U$ and $U'$. 
	Then
	$$ \inf_{0 \neq w \in gW} \frac{\|Lw\|}{\|w\|} \ge  \inf_{0 \neq w \in U} \frac{\|Lw\|}{\|w\|}  \inf_{0 \neq w \in gW} \frac{\|\mathrm{P} w\|}{\|w\|} $$
	for all elements $g \in K$.
\end{lemma}
\begin{proof}
	Consider a group  element $g \in K$.  
	Let $w \in g W$ be any non-zero vector. Write      $w = \mathrm{P}  w + \mathrm{P}' 
	w$. 
	Therefore 
	$ Lw = L\mathrm{P} w + L\mathrm{P}'  w$ with   $L\mathrm{P} w \in U$   and   $L\mathrm{P}'  w \in U'$.
Using the fact that   $\mathrm{P}w \in U $  we obtain
$$	
\|L w\| \ge \|L \mathrm{P} w\| \ge 
\|\mathrm{P}w \| \inf_{0 \neq w \in U} \frac{\|Lw\|}{\|w\|}
  \ge 	
\|w\|  \inf_{0 \neq w \in U} \frac{\|Lw\|}{\|w\|}  \inf_{0 \neq w \in gW} \frac{\|\mathrm{P} w\|}{\|w\|}.
$$
The  conclusion follows by taking the infimum over all such non-zero vectors $w$.
\end{proof}

The remainder of \S\ref{sec:grassmanians} is dedicated to setting up the necessary mechanism towards deducing Theorem \ref{thm:goodness: probability of q being small} as a consequence of Theorem \ref{thm:goodness on a compact group}.

\subsection*{Norms of projections} 
The statement of Theorem \ref{thm:goodness: probability of q being small} is independent of the choice of the norm $\|\cdot\|$ on the finite dimensional $k$-vector space $V$ up to modifying the value of the constant $\kappa_\mathfrak{N}$.  It will be convenient to assume for the remainder of the current \S\ref{sec:grassmanians} that the norm $\|\cdot\|$ is given as follows:

Fix a basis $\beta = \{e_1,\ldots,e_n\}$ for the $k$-vector space $V$ such that the subspace $U$ is spanned by $\{e_1,\ldots,e_{\dim_k U} \} $. We will assume that   $\|\cdot\|$ is the Euclidean   norm making the basis $\beta$    orthonormal in the real case or the maximum norm with respect to the basis  $\beta$ in the non-Archimedean case.

Recall that $l \in \NN$ is a fixed integer satisfying $l \le \dim_k U$.  
The exterior power $\bigwedge^l V$  is a  $k$-vector space satisfying $\dim_k \bigwedge^l V = {n \choose l}$. There is a    $k$-linear subspace   $\bigwedge^l U \le \bigwedge^l V$ and  a projection $\bigwedge^l \mathrm{P} : \bigwedge^l V \to \bigwedge^l U$.
Endow  the  vector space $\bigwedge^l V$ with the supremum norm $\|\cdot\|_\beta$ with respect to the fixed basis 
$$ \bar{\beta} = \{ e_{i_1} \wedge e_{i_2} \wedge \cdots \wedge e_{i_l} \: : \: 1 \le i_1 < i_2 < \cdots < i_l \le n \}.$$

Consider the   continuous function $q$ on the Grassmannian $ \mathfrak{Gr}(l,V)$ given by
\begin{equation}
\label{eq:qW}
 q(W) = \sup \{ \| \bigwedge^l \mathrm{P}  (w_1 \wedge \cdots \wedge w_l)  \|_\beta  \: : \: w_1,\ldots,w_l \in W \quad \text{and} \quad \|w_i\| \le 1  \}.
\end{equation}
for any $l$-dimensional $k$-linear subspace $W \le V$.
%
%
Note that $q(W) = 0$ if and only if $\mathrm{rank}_k   \mathrm{P}_{|W} < l$. This happens if and only if $\dim_k  (W \cap U')  > 0$. 



%

%
%
\begin{prop}
\label{prop:polynomial and norm on projections}
Every  subspace $W \le V$ the projection $\mathrm{P} : W \to U $ satisfies
\begin{equation}
\label{eq:bound on norm}
 \inf_{0 \neq w \in W} \frac{\|\mathrm{P} w\|}{\|w\|} \ge   q(W).
\end{equation}
\end{prop}
%


 Before proceeding, we  will need an upper bound on the absolute value of the determinant of a matrix in terms of, say, the norms of its column vectors.
\begin{lemma}
	\label{lem:determinant bound}
	Let $A$ be an $l$-by-$l$ matrix over the local field $k$ for some $l \in \NN$. Let   $v_i \in k^l$  denote the $i$-th column of the matrix $A$ for every  $i \in \{1,\ldots,l\}$.   Then
	\begin{equation}
	\label{eq:determinant bound}
	|\det A| \le   \prod_{i=1}^l \|v_i\| 
	\end{equation}
where $\|\cdot\| $ is the  Euclidean norm if  $k = \RR$  and for the supremum norm if  $k$ is non-Archimedean.
\end{lemma}
\begin{proof}[Proof of Lemma \ref{lem:determinant bound}]
The real case is the Hadamard determinant inequality. In   the non-Archimedean case Equation (\ref{eq:determinant bound}) follows immediately from the Leibniz determinant formula and the non-Archimedean triangle inequality.
\end{proof}

\begin{proof}[Proof of Proposition \ref{prop:polynomial and norm on projections}]
Fix a $k$-linear subspace $W \in \mathfrak{Gr}(l,V)$. Consider any   unit vector $w \in W$ and any  ordered basis $\omega=\{ w_1, \ldots, w_l\} $ for the subspace $W$ consisting of unit vectors. 
To establish Equation (\ref{eq:bound on norm}) it suffices to verify that
\begin{equation}
\label{eq:up to theta}
 \| \bigwedge^l \mathrm{P}  (w_1 \wedge \cdots \wedge w_l)  \|_\beta \le  \|\mathrm{P}w\|.
 \end{equation}

We  claim that there is a   linear transformation $B \in \mathrm{SL}(W)$  satisfying $B w_1 = xw$ for some $x \in k$ with $|x| \le 1$ and $\|Bw_i \| = \|w_i\| = 1$ for all $i \in \{2,\ldots,l\}$. The claim follows in the real case with $x = 1$ as the  $\mathrm{SO}(W)$-action  is transitive on Euclidean unit vectors in $W$. To prove the claim in the non-Archimedean case let $\mathcal{O}_k$ be the ring of units and $\pi \in \mathcal{O}_k$ a uniformizing element. Let   $i \in \NN \cup \{0\} $ be a minimal integer such that  $\pi^i w$ belongs to the $\mathcal{O}_k$-module spanned by $w_1,\ldots, w_l$. Write  $\pi^i w = \sum_{j=1}^l x_jw_j$ for some coefficients $x_j \in \mathcal{O}_k$.  The minimality of $i$ implies that $|x_j| = 1$ for some $j$. Assume without loss of generality  $|x_1| = 1$ and write
$$ (x_1^{-1}) \pi^i w -  w_1 = \sum_{j=2}^l x_1^{-1} x_j w_j.$$
By considering the suitable product of elementary matrices, we obtain the matrix $B \in \mathrm{SL}(W)$    taking the vector $w_1$ to $xw$ with $x= x_1^{-1} \pi^i$ and preserving all other vectors $w_j$. This completes the claim.

Now let $B \in \mathrm{SL}(W)$ be any  transformation as provided by the claim. Up to replacing the ordered basis $\omega$ with the ordered basis  $B \omega$  and observing that
$ w_1 \wedge \cdots \wedge w_l = Bw_1 \wedge \cdots \wedge Bw_l $
  we may assume  without loss of generality that $w_1  = xw$ with some $x\in k, |x| \le 1$ and maintain the assumption that $\|w_i \| = 1$ for all $i\in\{2,\ldots,l\}$.

The coordinates  of the $l$-vector $\bigwedge^l \mathrm{P} (w_1 \wedge \cdots \wedge w_l) \in \bigwedge^l V$  in   the   basis $\bar{\beta}$ are the determinants of the corresponding $l$-minors of the matrix representing the  
 linear transformation   $  \mathrm{P}_{|W} : W \to U$ with respect to the basis  $\omega$ for the subspace $W$ and the basis $e_1,\ldots,e_{\dim_k U}$ for the subspace $U $.  Let therefore $e_{i_1} \wedge \cdots \wedge e_{i_l} \in \bar{\beta}$ with $1 \le i_1 < \cdots < i_l \le n$  be a particular  basis element   
satisfying
$$   \| \bigwedge^l \mathrm{P}  (w_1 \wedge \cdots \wedge w_l)  \|_\beta = | \det A_{i_1,\ldots,i_l} | $$
where $ A_{i_1,\ldots,i_l}$ is the corresponding $l$-minor of the above-mentioned matrix.
The determinant bound established in Lemma \ref{lem:determinant bound} gives
$$   |\det A_{i_1,\ldots,i_l} |   \le \prod_{i=1}^{ l}\|\mathrm{P} w_i \| \le \|\mathrm{P}w_1 \| = x \|\mathrm{P} w\| \le \|\mathrm{P}w \|.$$
We have used the fact that $\mathrm{P}$ is contracting so that   $\|\mathrm{P}w_i\| \le 1$ for all $i \in \{1,\ldots,l\}$ as well as that $|x| \le 1$.
%
%
%
\end{proof}

\subsection*{Proof of the main result}

Let  $\mathfrak{N}$  be a   closed subset of the  Grassmannian $ \mathfrak{Gr}(l,V)$. Assume  that  every subspace $W \in \mathfrak{N}$ satisfies  
$ gW \cap U' = \{0\}$ for some   element $g \in K$.
Consider the family $\mathcal{F}(\mathfrak{N}) $   of $k$-analytic functions as introduced in Equation (\ref{eq:FN family}).  Denote $\|Q\|_K=\sup_{g\in K}\|Q(g)\|_\beta$ for every function $Q \in \mathcal{F}(\mathfrak{N})$.



\begin{prop}
	\label{prop:bound on norm of q}
	There is a constant $E > 0$ such that every subspace $W \in \mathfrak{N}$ satisfies
	$$ E \le \sup_{\substack{w_i \in W\\ \|w_i\| \le 1}}  \|Q_{w_1 \wedge \cdots \wedge w_l}\|_K \le 1. $$
\end{prop}
\begin{proof}
The existence of the upper bound follows from Lemma \ref{lem:determinant bound}. 
		To establish the existence of the lower bound, assume towards contradiction that there is a sequence of subspaces $W_i \in \mathfrak{N}$  such that 
	$$ \lim_i \sup_{\substack{w_j \in W_i\\ \|w_j\| \le 1}}  \|Q_{w_1 \wedge \cdots \wedge w_l}\|_K = 0.$$
 Up to passing to a subsequence, we may assume that $W_i \to W_0 \in \mathfrak{N} \subset \mathfrak{Gr}(l,V )$. 
 Consider an arbitrary ordered basis $\omega_0 = \{w_1,\ldots,w_l\}$ consisting of unit vectors for the limiting vector subspace $W_0$. There is a sequence  $\omega_i$ of ordered bases for $W_i$ consisting of unit vectors such that $\omega_i \to \omega_0$. It follows that $\|Q_{w_1\wedge\cdots\wedge w_l}\|_K = 0$. Since the basis $\omega_0$ was arbitrary it follows that $q(gW_0)  = 0$ for all elements $g \in K$ in the sense of Equation (\ref{eq:qW}). This is a contradiction to the above assumption.
\end{proof}

\begin{prop}
\label{prop:relatively compact}
For every constant $E > 0$ the subfamily 
$$
\mathcal{F}_{E}(\mathfrak{N}) = \{ Q  \in \mathcal{F}(\mathfrak{N}) \: : \: \|Q\|_K \ge E \} 
 $$ 
is compact in the topology of uniform convergence on the compact group $K$.
\end{prop}
\begin{proof}
Consider a sequence of functions $Q_i = Q_{w^i_1 \wedge \cdots \wedge w^i_l} \in \mathcal{F}_E(\mathfrak{N})$  where $ w^i_1,\ldots,w^i_l $ is an ordered basis  for some subspace $W_i \in \mathfrak{N}\subset \mathfrak{Gr}(l,V)$.
Assume that the sequence $Q_i$ uniformly converges to some continuous function $\Psi : K \to \bigwedge^l U$. Up to passing to a subsequence, we may assume that $W_i \to W_0 \in \mathfrak{N}$, that $w^i_j \to w^0_j$ for each $j \in \{1,\ldots,l\}$ and that $w^0_1,\ldots,w^0_l$ is an ordered $l$-tuple of unit vectors contained in $W_0$. 
Denote $Q_0 = Q_{w^0_1\wedge \cdots \wedge w^0_l}$. 
There are elements $g_i \in K$ with $\|Q_{i}(g_i)\|_\beta \ge E$. Up to passing to a further subsequence, we may assume that $g_i \to g_0 \in K$. Therefore $\|Q_{0}(g_0)\|_\beta \ge E$ and the ordered $l$-tuple $w^0_1,\ldots,w^0_l$ is linearly independent, hence a basis of $W_0$. In particular  $Q_{0} \in \mathcal{F}_E(\mathfrak{N})$. Finally it is clear that  the uniform limit $\Psi$ of the sequence $Q_{i}$ coincides with $Q_{0}$.
\end{proof}

We are   ready to complete the proof of the main result of \S\ref{sec:grassmanians}.

\begin{proof}[Proof of Theorem \ref{thm:goodness: probability of q being small}]


Let $\mathcal{F}(\mathfrak{N})$ be the family of polynomial mappings associated to the   subset $\mathfrak{N}   $ as introduced in Equation (\ref{eq:FN family}). 
Fix a sufficiently small constant $E > 0$ as provided by Proposition \ref{prop:bound on norm of q}. Namely every subspace $W \in \mathfrak{N}$ admits some ordered basis $ w_1,\ldots,w_l$   of unit vectors such that $Q_{w_1\wedge\cdots\wedge w_l}  \in \mathcal{F}_E(\mathfrak{N})$. 
The family  $\mathcal{F}_E(\mathfrak{N})$ of $k$-analytic maps from the compact group $K$ to the $k$-vector space $\bigwedge^l U$ is compact in the topology of uniform convergence  according to Proposition \ref{prop:relatively compact}.


The Haar measure $\eta_K$ belongs to the canonical measure class of $K$ regarded as a $k$-analytic manifold.
Taking into   account Proposition \ref{prop:polynomial and norm on projections}, the desired  Equation (\ref{eq:desired estimate}) follows from Theorem \ref{thm:goodness on a compact group} providing   estimates on the  measures of sublevel sets of $k$-analytic functions with respect to the  Haar    measure $\eta_K$ on the compact $k$-analytic group $K$.
\end{proof}

\section{Nilpotent subalgebras of   semisimple Lie algebras}
\label{sec:nilpotent}

Let $k$ be either the field $\RR$ or a non-Archimedean local field with absolute value $|\cdot|$. Let $\GG$ be a connected   simply-connected semisimple $k$-algebraic linear group without $k$-anisotropic factors. 
Denote $G = \GG(k)$ so that  $G$ is a $k$-analytic group.

In the positive characteristic case we   make the  additional   assumption  that $\mathrm{char}(k)$ is a good prime for the semisimple group $\GG$ in the sense of   \cite[Definition 4.1]{springer1970conjugacy}. This means that $\mathrm{char}(k)$ does not divide the coefficient of any simple root in the highest root. We point out that the only primes which fail to be good for some semisimple group are $2,3$ and $5$.
\subsection*{Maximal split torus}

Let $\TT$ be a maximal $k$-split torus of $\GG$. Let  $\Phi = \Phi(\TT,\GG)$ be the relative $k$-root system of the group $\GG$ with respect to the torus $\TT$. 

Choose an ordering on the $k$-root system $\Phi$ and let $\Phi^+$ and $\Phi^- = - \Phi^+$  denote the subsets of positive and negative roots respectively. Let $\Phi_{ 0}^+ \subset \Phi^+$ denote the subset of the simple positive roots. 

Let $\BB$ be the minimal parabolic $k$-subgroup corresponding to  $\Phi^+$. 
We have $\BB = Z(\TT) \UU$ where $\UU$ is the unipotent radical of $\BB$. Every   minimal parabolic $k$-subgroup is conjugate to $\BB$ via an element of $G$  \cite[\S I.21.12]{borel2012linear}. 
 The subgroup $\UU$ is defined over $k$ \cite[\S V.21.11]{borel2012linear}. 

Fix notations for the corresponding groups of rational points, namely
\begin{equation}
B = \BB(k), \quad T = \TT(k)  \quad \text{and} \quad U = \UU(k).
\end{equation}

\subsection*{The relative Weyl group}

Let $W = W(\TT,\GG)$ be the relative Weyl group    with respect to the maximal $k$-split torus   $\TT$  in the group   $\GG$. It is a finite Coxeter group containing an involution  $s_\alpha$  for each root $\alpha \in \Phi$  \cite[\S21]{borel2012linear}.  In fact the set involutions  $\Sigma = \{s_\alpha\}_{\alpha \in \Phi^+_0}$ corresponding to the simple roots $\Phi^+_0$ generates the Weyl group $W$. 

Let $L_\Sigma : W \to \NN \cup \{0\}$ be the length function on the relative Weyl group $W$ with respect to the generating set $\Sigma$. Let $w_0 \in W$ be the longest element  of $W$, i.e. the unique element satisfying $L_\Sigma(w_0) \ge L_\Sigma(w)$ for every other element $w \in W$  \cite[\S1.8]{humphreys1990reflection}. It is the unique element satisfying  $w_0 \Phi^+ = \Phi^-$ and such   that $\BB^{w_0}$ is the opposite parabolic $k$-subgroup to $\BB$. The longest element $w_0$ is independent of the choice of ordering.

\subsection*{The compact subgroup $K$}

Let $K$ be a maximal compact subgroup of the $k$-analytic group $G$. In the real case assume that $K$ is  compatible with the Cartan involution. In the positive characteristic case assume that $K$ is a \emph{good} maximal compact subgroup, namely it admits a representative of every element of the spherical Weyl group.  
In either case, let $\eta_K$ denote the   Haar measure on the compact group $K$ normalized so that $\eta_K(K) = 1$. 

The group $G$ admits an Iwasawa decomposition $G = K B$  \cite{borel1965groupes}. In particular the compact subgroup $K$ is transitive in its action on the homogeneous space $G/B$. 

\subsection*{Unipotent subgroups --- positive characteristic case}

We state an important theorem concerning unipotent subgroups of simply-connected semisimple groups in positive characteristic.

\begin{theorem}[Gille]
\label{thm:Gille}
Assume that  the characteristic of $k$ is a good prime for $\GG$. Then every unipotent subgroup  $V$ of $G$ is contained in the group of $k$-points of some   minimal parabolic $k$-subgroup.
\end{theorem}
\begin{proof}
 This   follows from   the assumption that $\GG$ is simply connected and  relying on the work  \cite{gille2002unipotent}. See also  \cite[Theorem 2.7]{golsefidy2009lattices}.  
 \end{proof}

%

 

\subsection*{Intersections with   minimal parabolic $k$-subgroups}

We show that any torus $S$ in $G$ as well as the unipotent subgroup $U$  admit  a zero-dimensional intersection with a generic   minimal parabolic $k$-subgroup.

The proofs are based to a   large extent on the ideas of \cite[Lemma 3]{kazhdan1968proof}.
In particular, we   rely on the Zariski density of $k$-points --- namely if $\HH$ is any connected reductive linear algebraic $k$-group then $\HH(k)$ is   Zariski dense, see \cite[18.3]{borel2012linear}  or   \cite[\S 2.1 Theorem 2.2]{platonov1992algebraic} in the real case. Therefore the $k$-dimension of the group  of $k$-points $\HH(k)$ is equal to the    dimension of $\HH$.  

\begin{prop}
 \label{prop:intersection of unipotent and Borel}
 The subset
$$ \{g \in G \: : \: \dim (U^g \cap B) = 0 \}$$ 
is  non-empty and Zariski open in $G$.
\end{prop}
\begin{proof}
By the Zariski density of  $k$-points   it  will suffice  to show that 
$$ \Psi = \{g \in \GG \: : \: \dim (\UU^g \cap \BB) = 0 \}$$ 
is a  non-empty   Zariski open subset of $\GG$. The longest element $w_0$ of the relative Weyl group satisfies $\BB^{w_0} \cap \BB = \TT$. Therefore $\dim(\UU^{w_0} \cap \BB) = 0$ and in particular    $\dim(\UU^{uw_0b} \cap \BB) = 0$  for every pair of elements $u \in \UU$ and $b \in \BB$. In other words, the   big Bruhat cell $\UU w_0 \BB  $ is contained in $\Psi$. This Bruhat cell is Zariski open and non-empty.
 \end{proof}

 \begin{prop}
 	\label{prop:intersection of torus and Borel}
 	Let $S = \SSS(k)$ where $\SSS \le \GG$ is any $k$-torus.  Then the subset
 	$$ \{g \in G \: : \: \dim (S^g \cap B) = 0 \}$$ 
 	is non-empty and  Zariski open in $G$.
 \end{prop}
 
 We do not assume that the $k$-torus $\SSS$ is split.
 
 \begin{proof}[Proof of Proposition \ref{prop:intersection of torus and Borel}]
 	By the preceding discussion about   the Zariski density of $k$-points  it will suffice to show that 
 	$$ \Omega =  \{g \in \GG \: : \: \dim (\SSS^g \cap \BB) > 0 \}$$ 
 	is  a  proper Zariski closed subset of $\GG$. 
 	Up to conjugation by an element of $\GG$ we may assume   that $\SSS \le \TT$ \cite[\S IV.11.3]{borel2012linear}. 
 	
 	Consider any element   $g \in \Omega$. There is a Zariski closed subgroup $\SSS_1 \le \SSS$ depending on   $g$ such that   $\dim \SSS_1 > 0$ and  $\SSS_1^g \le \BB$. Since all maximal tori in the solvable group $\BB$ are conjugate   there is an element $h \in \BB$ with $\SSS_1^{gh} \le \TT$ \cite[\S III.10.16]{borel2012linear}. As $\SSS_1 \le \TT$ this condition implies  that $gh \in  \mathrm{Z}_{\GG}(\SSS_1)$    \cite[\S III.8.10 Corollary 1]{borel2012linear}. 
 	We conclude that
 	$$ \Omega \subset \bigcup_{\substack{\SSS_1 \le \SSS\\ \dim \SSS_1 > 0}} \mathrm{Z}_{\GG}(\SSS_1)\BB.
 	$$
 	Our    assumption that $\GG$ has no $k$-anisotropic factors implies that each  $ \mathrm{Z}_{\GG}(\SSS_1) \BB$ is a proper standard parabolic $k$-subgroup    \cite[\S V.21.11]{borel2012linear}. There are only finitely many such   subgroups.  In particular $\Omega$ is  proper and Zarsiki closed.
 \end{proof}

\subsection*{  Lie   algebras}
 Let $\mathfrak{g} = \mathrm{Lie}(G)$ denote the Lie algebra\footnote{In the Archimedean case, the Lie algebra $\mathfrak{g}$ taken in the sense of Lie theory coincides with the real points of the Lie algebra of $\GG$ taken in the sense of algebraic groups \cite[Lemma \S3.3.1]{platonov1992algebraic}.} of the Lie group $G$. 
 Let $\mathfrak{g}_\alpha$ be the relative $k$-root space corresponding to  each relative $k$-root $\alpha \in \Phi$, namely
 \begin{equation}
 \label{eq:weights}
  \Ad{\sse} X   = \alpha(\sse) X \quad \forall \sse \in T, \quad \forall X \in \mathfrak{g}_\alpha.
  \end{equation}
   Consider the two Lie subalgebras
 $$\mathfrak{u}^- = \bigoplus_{\alpha \in \Phi^-} \mathfrak{g}_\alpha \quad \text{and} \quad \mathfrak{u}^+= \bigoplus_{\alpha \in \Phi^+} \mathfrak{g}_\alpha.$$
The Lie algebra $\mathfrak{g}$ admits a direct sum decomposition
 $$ \mathfrak{g} =\mathfrak{b} \oplus \mathfrak{u}^-     \quad \text{and} \quad \mathfrak{b} = \mathfrak{u}^+ \oplus \mathfrak{g}_0  $$ 
where  $\mathfrak{b} = \mathrm{Lie}(B)$ and   $\mathfrak{g}_0 = \mathrm{Lie}(Z_G(T))$  \cite[\S V.21.7]{borel2012linear}.
Fix some $\Ad{K}$-invariant norm   $\|\cdot\|$  on the Lie algebra $\mathfrak{g}$. In the real case assume that this norm is coming from an $\Ad{K}$-invariant inner product on $\mathfrak{g}$. 

\subsection*{Nilpotent Lie subalgebras}

Let  $\mathfrak{N}(\mathfrak{g})$ denote the subset of the Grassmannian $\mathfrak{Gr}(\mathfrak{g})$ consisting of all the    nilpotent Lie subslagebras of $\mathfrak{g}$. 

The subspace of $\mathfrak{Gr}(\mathfrak{g})$ consisting of Lie subalgebras is closed. This follows from the fact that the Lie bracket operation $\left[\cdot,\cdot\right] : \mathfrak{Gr}(\mathfrak{g}) \times \mathfrak{Gr}(\mathfrak{g}) \to \mathfrak{Gr}(\mathfrak{g})$ is lower semi-continuous.  By induction on the nilpotency degree we deduce moreover that $\mathfrak{N}(\mathfrak{g})$ is a closed subset of $\mathfrak{Gr}(\mathfrak{g})$.

There is an alternative argument  in the real case. Namely,  recall that  $X \in \mathfrak{g}$ is called \emph{ad-nilpotent} if $\mathrm{ad}(X)^{\dim \mathfrak{g}} = 0$. Engel's theorem says that a Lie subalgebra $\mathfrak{h} \le \mathfrak{g}$ is nilpotent if and only if every $X \in \mathfrak{g}$ is ad-nilpotent. The adjoint map $\mathrm{ad} : \mathfrak{g} \to \mathrm{End}(\mathfrak{g})$ is continuous. Therefore being ad-nilpotent is a closed condition in $\mathfrak{g}$, and it follows that $\mathfrak{N}(\mathfrak{g})$  is closed in $\mathfrak{Gr}(\mathfrak{g})$.

We remark that $\mathfrak{N}(\mathfrak{g})$ is   Zarsiki closed  in $\mathfrak{Gr}(\mathfrak{g})$ but we will not need this.

\subsection*{Algebraic Lie subalgebras} 

A Lie subalgebra $\mathfrak{h} \le \mathfrak{g}$ is \emph{algebraic} if $\mathfrak{h} = \mathrm{Lie}(\HH(k))  $ for some algebraic subgroup $\HH$ of $\GG$. For more on this notion see \cite[\S I.7]{borel2012linear}.
 
 We rely on   Propositions  \ref{prop:intersection of unipotent and Borel} and \ref{prop:intersection of torus and Borel}
to deduce the following statement concerning almost every conjugate of an algebraic   nilpotent Lie subalgebra.

\begin{prop}
	\label{prop:conjugates of nilpotent}
	If $\mathfrak{n} \in \mathfrak{N}(\mathfrak{g})$ is an algebraic nilpotent Lie subalgebra then   the subset
 	$$ \{g \in G \: : \: \Ad{g} \mathfrak{n} \cap \mathfrak{b} = \{0\} \} $$ 
 	is  non-empty and Zariski open in $G$.
\end{prop}
\begin{proof}
	Let $\mathfrak{n}$ be an algebraic nilpotent Lie subalgebra of $\mathfrak{g}$. 
	Let $N$ be a Zariski closed nilpotent subgroup of $G$ with $  \mathrm{Lie}(N) = \mathfrak{n} $. 
	The nilpotent subgroup $N$  decomposes as a direct product $N = N_s \times N_u$ where $N_s$ and $N_u$ are the semi-simple and unipotent parts of $N$, respectively \cite[10.6]{borel2012linear}. Both $N_s$ and $N_u$ are defined over $k$. 
	The subgroup $N_s$ is a $k$-torus.
	
	The intersection $N  \cap B^g$ is nilpotent and Zariski closed for every element $g\in G$. By the same reasoning as above, these intersections admit    direct product decompositions
	$$N \cap B^g = N_s(g) \times N_u(g)$$ 
	for some  Zariski closed subgroups $N_s(g) \le N_s$ and $N_u(g) \le N_u$ depending on the element  $ g\in G$.
	By Proposition \ref{prop:intersection of unipotent and Borel} and  Proposition \ref{prop:intersection of torus and Borel}, $\dim N_s(g) = \dim N_u(g) =0$ for every element $g$ belonging to   a non-empty Zariski open subset of $G$. The result follows.
	\end{proof}

\subsection*{Nilpotent Lie subalgebras --- real case} Assume that the local field $k$ is $\RR$ in  the following two Propositions \ref{prop:every nilpotent subsalgebra is contained in an algebraic nilpotent subalgebra} and \ref{prop:order of adjoint in real case}.

 \begin{prop}
\label{prop:every nilpotent subsalgebra is contained in an algebraic nilpotent subalgebra}
Every nilpotent Lie subalgebra $\mathfrak{n} \in \mathfrak{N}(\mathfrak{g})$ is contained in some algebraic nilpotent subalgebra.
\end{prop}
\begin{proof}
 Let $\mathfrak{n}$ be a nilpotent Lie subalgebra of $\mathfrak{g}$.   Lie's third theorem provides us with a    simply connected Lie group $N_0$ with Lie algebra $ \mathrm{Lie}(N_0) \cong \mathfrak{n}$. There is a Lie group homomorphism $f : N_0 \to G$ such that the differential $df : \mathrm{Lie}(N_0) \xrightarrow{\simeq} \mathfrak{n}$  is an isomorphism. Let $N$ denote the Zariski closure of the nilpotent subgroup $f(N_0)$ inside $G$. The  subgroup $N$ is nilpotent and Zariski closed. Its Lie algebra  is algebraic by definition and contains $\mathfrak{n}$.
\end{proof}

Let $\mathrm{rank}(K)$ denote the \emph{compact rank} of the maximal compact group $K$, in other words $\mathrm{rank}(K)$ is the  dimension of any maximal compact torus of $K$. Let $\mathrm{P}$ denote be a  fixed  projection from the Lie algebra $\mathfrak{g}$ to its Lie subalgebra $\mathfrak{u}^-$. 

\begin{prop}
\label{prop:order of adjoint in real case}

  Let $\mathfrak{n} \in \mathfrak{N}$ be any nilpotent Lie subalgebra  with  basis $\beta_\mathfrak{n}$. Then the real analytic function $q_\mathfrak{n} \in \mathcal{A}(G, \bigwedge^{\dim_\RR \mathfrak{n}} \mathfrak{u}^-)$ given by
\begin{equation}
\label{eq:q_n}
 \quad q_\mathfrak{n}(g) =    \bigwedge_{X \in \beta_\mathfrak{n}} (\mathrm{P}  \circ   \mathrm{Ad} ) (g) X \quad \forall g\in G
 \end{equation}
 satisfies $\mathrm{ord}_{K} q_\mathfrak{n} <  (6  \hgt{\mathfrak{g}}  \dim_\RR \mathfrak{n}+1)^{\mathrm{rank}(K)}$.
\end{prop}
\begin{proof}
The real analytic function $q_\mathfrak{n}$ is non-zero on a Zariski open subset of $G$ as can be seen by combining Proposition \ref{prop:conjugates of nilpotent} and Proposition \ref{prop:every nilpotent subsalgebra is contained in an algebraic nilpotent subalgebra}. Note that for any pair of elements $g \in G$ and $b \in B $ the function $q_\mathfrak{n}$ satisfies  $q_\mathfrak{n}(g) = 0$ if and only if $q_\mathfrak{n}(bg) = 0$. The Iwasawa decomposition $G = K B $  shows that the function $q_\mathfrak{n}$ is not identically zero on each connected component of the compact   subgroup $K$. 

 We wish to bound the order $\mathrm{ord}_K q_\mathfrak{n}$ relying on Lemma \ref{lem:order for wedge product}. With this goal in mind, let $T_K$ be a maximal compact torus of $K$ and    $\Lambda_\text{analytic}$ be the lattice of analytically integral forms on $\mathrm{Lie}(T_K)$. Note that $\Lambda_\text{analytic} \cong \ZZ^{\mathrm{rank}(K)}$.

  Let $S = \SSS(\RR)$ be a maximal $\RR$-torus containing the compact torus $T_K$ and $\Delta$ be the absolute root system associated to $S$  in the semisimple group $G$.  The restrictions to $T_K$  of the simple positive roots (with respect to some notion of positivity on $\Delta$) are analytic and are a $\ZZ$-basis for some lattice $\Lambda \le \Lambda_\text{analytic}$, see Proposition \ref{prop:restricting to maximally compact} of the appendix. 
 It follows that   the weights of the $T_K$-action on the complexification of the vector space $\bigwedge^{\dim_\RR \mathfrak{n}} \mathfrak{g}$ are contained in a ball of radius $ \hgt{\Delta} \dim_\RR \mathfrak{n}$ in the lattice $\Lambda$.  
 
 A comparison of the possibilities for absolute and   restricted root systems in the real case \cite[Reference Chapter, Table 9]{onishchik2012lie} shows\footnote{ If an absolute root system $\Psi$ is classical then each real restricted root system $\Phi$ obtained from it is classical as well. More generally $\hgt{\Psi} \le 2 \hgt{\Phi}$ unless $\Psi$ and $\Phi$ have types E6 and F4 respectively.} that $\hgt{\Delta} \le 3 \hgt{\Phi} = 3\hgt{\mathfrak{g}}$. 
The size of the ball of radius  $3\hgt{\mathfrak{g}} \dim_\RR \mathfrak{n}$ in the lattice $\Lambda \cong \ZZ^{\mathrm{rank}(K)}$ is at most $(6  \hgt{\mathfrak{g}}  \dim_\RR \mathfrak{n}+1)^{\mathrm{rank}(K)}$. This concludes the proof.
\end{proof}

\subsection*{Nilpotent Lie subalgebras --- positive characteristic case}

Assume that the local field $k$ has positive characteristic. 
Let $\mathrm{P}$ be a    projection from the Lie algebra $\mathfrak{g}$ to its Lie subalgebra $\mathfrak{u}^-$.  
We establish the following the following order estimate.

\begin{prop}
\label{prop:order in pos char case}
Let $\beta^+$ be a $k$-basis for $\mathfrak{u}^+$.   
 Then the $k$-analytic function 
$$q_\mathfrak{u} \in \mathcal{A}(K, \bigwedge^{\dim _k \mathfrak{u}^-} \mathfrak{u}^-) $$ 
 given by
$$q_\mathfrak{u}(g) =    \bigwedge_{X \in \beta^+} (\mathrm{P}  \circ   \mathrm{Ad} ) (g) X    \quad \forall g \in K $$ 
satisfies 
$$ \mathrm{ord}_K  q_\mathfrak{u} \le 4  \hgt{\mathfrak{g}}  \dim^3_k U \dim_k \mathfrak{g}.$$
\end{prop}

Some preparation is needed before giving the  proof of Proposition \ref{prop:order in pos char case}. To begin with, recall that the relative $k$-root system $\Phi$ may not be reduced in general.
For each   relative $k$-root $\alpha \in \Phi$ denote
\begin{equation}
\bar{\alpha}  =   \begin{cases}    2\alpha& 2\alpha \in \Phi, \\ \alpha  & 2\alpha \notin \Phi \end{cases} 
\end{equation} 
and consider the subset $(\alpha) \subset \Phi$ given by
\begin{equation}
(\alpha) = \NN \alpha \cap \Phi = \{\alpha, \bar{\alpha} \} = 
 \begin{cases} \{\alpha, 2\alpha\}& 2\alpha \in \Phi, \\\{\alpha\} & 2\alpha \notin \Phi.
  \end{cases} 
\end{equation} 
Generally speaking, the group $G$ admits a metabelian unipotent $k$-subgroup $ U_{(\alpha)}$ whose center is the connected unipotent $k$-subgroup   $U_{(\bar{\alpha})}$.  The groups $ U_{(\alpha)}$ and $U_{(\bar{\alpha})}$ are normalized by the maximal $k$-split torus $T$ \cite[Proposition 21.9]{borel2012linear}.

The following result uses the notion of strong order introduced in Definition \ref{def:strong order}.
\begin{prop}
\label{prop:ord of Ad on unipotent}
Let $\alpha \in \Phi$ be a relative $k$-root. Then the adjoint representation of the group $U_{(\bar{\alpha})}$ regarded as a  $k$-analytic map $$\mathrm{Ad} \in \mathcal{A}(U_{(\bar{\alpha})}, \mathrm{End}(\mathfrak{g}))$$ has strong order at most $4  \hgt{\mathfrak{g}} \dim_k U_{(\bar{\alpha})} \dim_k \mathfrak{g}$ at each point of $U_{(\bar{\alpha})}$.
%
\end{prop}
\begin{proof}
The connected   unipotent $k$-subgroup $U_{(\bar{\alpha})}$ is abelian and admits  a   $k$-isomorphism $\theta_\alpha : k^n \to U_{(\bar{\alpha})}$ where $n = \dim_k U_{(\bar{\alpha})}$ \cite[Lemma 21.17]{borel2012linear}. 
Since the adjoint representation $\mathrm{Ad} : G \to \mathrm{GL}(\mathfrak{g})$   is a central isogeny onto its image,  the $k$-rational representation  $\mathrm{Ad} \circ \theta_\alpha : k^n \to \mathrm{GL}(\mathfrak{g})  $ must be     a $k$-isomorphism onto its image \cite[Proposition 22.4]{borel2012linear}. It follows   that the map  $\mathrm{Ad} \circ \theta_\alpha \in \mathcal{A}(k^{n },\mathrm{End}(\mathfrak{g}))$ is an immersion of $k$-analytic manifolds \cite[\S I.2.5.3] {margulis1991discrete}.

We wish to conclude relying on Lemma \ref{lem:orbit map for unitpotent subgroup}. Indeed, the subgroup $\Ad{T}$ is a $k$-split torus of $\mathrm{GL}(\mathfrak{g})$  \cite[Corollary 8.4]{borel2012linear}. It satisfies
$$ \Ad{t} \Ad{\theta_\alpha(z)} \Ad{ t^{-1}} = \Ad{t \theta_\alpha(z) t^{-1}} = \Ad{ \theta_\alpha( \bar{\alpha}(t) z)} \quad \forall t \in T, z \in k^n.$$ 
The diagonal entries of the subgroup $\Ad{T}$ are given by   characters corresponding to the restricted roots $\Phi \cup \{0\}$. In particular, the condition $I \bar{ \alpha} = \beta - \beta'$ for some pair of roots $\beta,\beta' \in \Phi \cup \{0\}$ and some $I \in \NN$ implies that $I \le 2\hgt{\mathfrak{g}}$. 
 \end{proof}

The next Lemma concerns the combinatorics of the Bruhat cells. We use the set of generators $\Sigma$ as well as the length function $L_\Sigma : W \to \NN \cup \{0\}$ on the relative Weyl group $W$ introduced above. Recall that $w_0$ is   the longest element of $W$.

\begin{lemma}
\label{lem:conjugating into big cell}
Let $w   $ be any    element of the relative Weyl group $W$. Write
\begin{equation}
 w^{-1} w_0   =   s_{1} \cdots s_{ L_\Sigma( w^{-1} w_0 )}  
\end{equation}
as a reduced word, where  the $s_i = s_{\alpha_i}$ are generators belonging to  $\Sigma$.    Then  
\begin{equation}  Bw  B u_1\cdots u_n \subset Bw_0B
\end{equation}
for any choice of \emph{non-trivial} unipotent elements $u_i \in U_{(-\alpha_i)}$.
\end{lemma}
\begin{proof}
As  $w_0$ is the longest element   it follows  that $   L_\Sigma( w^{-1} w_0) =L_\Sigma(w_0) - L_\Sigma(w)$ \cite[p. 16, Equation (2)]{humphreys1990reflection}. Therefore
 $$ L_\Sigma(ws_{1} \cdots s_{i+1})  >L_\Sigma(ws_{1} \cdots s_{i}) \quad \forall i \in \{1,\ldots, L_\Sigma( w^{-1} w_0)-1\}.$$

For any pair of elements $w' \in W$ and $s \in \Sigma $ the condition $L_\Sigma(w's) > L_\Sigma(w')$  implies that $Bw'BsB \subset Bw'sB$ 
\cite[Proposition 21.22]{borel2012linear}. Observe that 
 $$ U_{(- \alpha)} \setminus \{e\} \subset Bs_{ \alpha } B$$
 for all simple positive roots $\alpha \in \Phi^+_0$, see \cite[Theorem 21.15]{borel2012linear} and its proof.  Applying these facts  inductively   shows that
 \begin{align*}
 Bw   u_1\cdots u_{i-1} B u_{i}  \subset Bw u_1 \cdots u_{i} B
 \end{align*}
for all $i \in \{1,\ldots,L_\Sigma(w^{-1}w_0) \}$ provided that the unipotent elements $u_i \in U_{(-\alpha_i)}$ are all non-trivial. The final induction step gives
$$ Bw B u_1\cdots    u_{ L_\Sigma( w^{-1} w_0 )}    \subset Bw u_1 \cdots u_{ L_\Sigma( w^{-1} w_0 )}  B = B w (w^{-1} w_0) B = Bw_0 B$$
as required.
\end{proof}

We are ready to complete the following argument.

\begin{proof}[Proof of Proposition \ref{prop:order in pos char case}]
Fix an arbitrary element $g \in K$. 
 There is a unique relative Weyl group element $w(g) \in W$ satisfying $g \in Bw(g) B$ \cite[Theorem 21.15]{borel2012linear}.   Denote $$w' = w'(g) = w(g)^{-1} w_0  $$ where $w_0 \in W$ is the longest element of $W$. Further denote $L = L_\Sigma(w'(g))$.  Express the Weyl group element $w' $ as a reduced word  $w'    =   s_{1} \cdots s_ {L }  $  where each $s_i $ belongs to the generating set  $\Sigma$ with $s_i = s_{\alpha_i}$ for some root $\alpha_i \in \Phi^+_0$. 


For each relative $k$-root $\alpha \in \Phi$ recall that $U_{ ( \bar{\alpha})}$ is a connected $k$-unipotent group $k$-isomorphic to a $k$-vector space. Let $\mathrm{Id}_\alpha, \mathrm{Fr}_\alpha : U_{ ( \bar{\alpha})} \to U_{ ( \bar{\alpha})}$  be the identity and the Frobenius maps, respectively.  For each index
$i \in \{1,\ldots,L \}$   define   a map\footnote{Strictly speaking, the objects $f_{g,i}, \Phi_g, H_g$ and $F_g$ all depend on the reduced word $w'(g)$ rather than on the element $g$.} 
$f_{g,i} : U_{ (- \bar{\alpha}_i)} \to U_{ ( -\bar{\alpha}_i)}$   by
$$ f_{g,i} = 
\begin{cases} 
\mathrm{Id}_{\alpha_i} & \text{$s_j \neq s_i$ for all $j < i$}, \\
 \mathrm{Fr}_{\alpha_i} & \text{otherwise}.\\
\end{cases}
$$
Consider the subset of simple positive roots
$$\Phi_g = \{ \alpha \in \Phi^+_0  :  \text{$s_i = s_{\alpha_i}$ for some $1 \le i \le L$} \}$$ and the  $k$-analytic product group $H_g =  \prod_{ \alpha \in \Phi_{g }} U_{ ( -\bar{\alpha})}$. We are interested in the 
 $k$-analytic map $F_g : H_g \to \mathrm{GL}(\mathfrak{g})$ given by
$$ F_g : (u_\alpha)_{\alpha \in \Phi_{g}}
 \mapsto
 g   f_{g,1} (u_{\alpha_1}) \cdots f_{g,L}(u_{\alpha_{L }}).$$
 
 The differential of  the map   $\mathrm{Ad} \circ F_g$ at the identity element $e \in H_g$ can be computed relying on  the chain-rule formula    \cite[\S3.2]{borel2012linear}. Namely
$$ \mathrm{d} \left(\mathrm{Ad} \circ F_{g}\right)_{|e}(X_\alpha )_{\alpha \in \Phi_g} = \Ad{g} \sum_{\alpha \in \Phi_g} \mathrm{ad}(X_\alpha).$$
 Note that $\mathrm{ad}(X_\alpha) \subset \mathrm{ad}( \mathfrak{g}_{-\bar{\alpha}})$ for all roots $\alpha \in \Phi_g$. Therefore the differentials  $\mathrm{ad}(X_\alpha)$ are linearly independent provided that the $X_\alpha$'s are non-zero. 
It follows that  the map $\mathrm{Ad} \circ F_g$ is a $k$-analytic immersion at the identity element of $H_g$.
 
 Let $H_0 \le H $  be a sufficiently small open subgroup such that   $(\mathrm{Ad} \circ F)_{g|H_0}$  is a $k$-analytic embedding into the compact subgroup $\Ad{K}$.  
 We know from Lemma \ref{lem:conjugating into big cell}
 that every element of $F_g(H_0) $ obtained as a product of non-trivial elements in each coordinate of $H_0$ belongs to the big Bruhat cell $Bw_0 B$. In particular $q_{\mathfrak{u}|F_g (H_0) }$ is not identically zero according to Proposition \ref{prop:intersection of unipotent and Borel} and its proof.

The $k$-analytic maps $\mathrm{Ad} \circ f_{g,i} : U_{(-\bar{\alpha}_i)} \to \mathrm{End}(\mathfrak{g})$ have strong order at most $4  \hgt{\mathfrak{g}} \dim_k U_{(\bar{\alpha}_i)} \dim_k \mathfrak{g}$ at each point of their domain   and for all $i \in \{1,\ldots,L\}$, see  Proposition  \ref{prop:ord of Ad on unipotent}.
As the map $(\mathrm{Ad} \circ F)_{g|H_0}$ is a $k$-embedding  it has strong order at most 
$$ \sum_{i=1}^L 4  \hgt{\mathfrak{g}} \dim_k U_{(\bar{\alpha}_i)} \dim_k \mathfrak{g} \le   4  \hgt{\mathfrak{g}} \dim^2_k U \dim_k \mathfrak{g} $$
at each point of $H_0$. The above computation relies on the fact that 
$$
 L  \le L_\Sigma(w_0) = |\Phi^+| \le \dim_k U.
$$
 
To conclude we estimate the order of the map $q_\mathfrak{u}$ at the element $g \in K$. Denote $  \dim _k \mathfrak{u}^+ = |\beta^+| = l$ and enumerate  $\beta^+ = \{ X_1,\ldots, X_l \}$. Note that
 $$(q_\mathfrak{u} \circ F_g) ( h) = \det \begin{pmatrix} 
 \vdots & & \vdots \\ 
 (\mathrm{Ad} \circ F_g)(h) X_{ 1} & \cdots &     (\mathrm{Ad} \circ F_g)(h) X_{ l}  \\
 \vdots& & \vdots
 \end{pmatrix} \quad \forall h \in H_0.
 $$
In light of  Lemma \ref{lem:strong order bounds order} we therefore have that
$$ \mathrm{ord}_g q_\mathfrak{u} \le \mathrm{ord}_{g} q_{\mathfrak{u}|F(H_0)} \le 4  \hgt{\mathfrak{g}} \dim^3_k U \dim_k \mathfrak{g}. $$
Since the point $g \in K$ was arbitrary this concludes the proof.
\end{proof}

\subsection*{Conjugation by semisimple elements}
Let $\sse \in T$ be any element. The linear operator $\Ad{\sse}$ as well as its inverse $\Ad{\sse^{-1}}$ preserve the direct sum decomposition $\mathfrak{g} =  \mathfrak{u}^- \oplus \mathfrak{b}$. It is   easy to compute the norms of the restrictions of these operators to, say,  $\mathfrak{u}^-$.

\begin{prop}
\label{prop:the norm and inverse of norm of a semisimple element}
Let $\sse \in T$ be such that $|\alpha(\sse)| < 1$ for all $\alpha \in \Phi^+$.     Then 
\begin{equation}
\label{eq:contraction and expansion}
\|\Ad{\sse}  \| =  \|\Ad{\sse^{-1}}  \| = \max_{\alpha\in\Phi^-} |  \alpha(X) |
\end{equation}
and
\begin{equation}
\label{eq:contraction}
  \|\Ad{\sse^{-1}}_{|\mathfrak{u}^-} \|  =  \frac{1}{\min_{\alpha\in\Phi^-} | \alpha(X) |}.
\end{equation}
\end{prop}

\begin{proof}
The operator $\Ad{\sse}$ is diagonizable, and the corresponding weights for its action on the Lie  subalgebra $\mathfrak{g}_\alpha$ are given by the relative $k$-root $\alpha \in  \Phi$, see  Equation (\ref{eq:weights}).
Therefore 
$$ \|\Ad{\sse}  \| = \max_{\alpha\in\Phi}  | \alpha(\sse)|
=  \max_{\alpha\in\Phi^-} | \alpha(\sse) |.$$
This implies that Equation (\ref{eq:contraction and expansion}) holds. 
Since $ \Phi = -\Phi$, the same reasoning shows that $\|\Ad{\sse}^{-1}  \| = \|\Ad{\sse} \| $.
The linear isomorphism $\mathrm{Ad}(\sse)$ preserves $\mathfrak{u}^-$. Therefore
$$  \|\Ad{\sse^{-1}}_{|\mathfrak{u}^-} \| =   \max_{\alpha\in\Phi^-} \frac{1}{| \alpha(\sse)|}  = \frac{1}{\min_{\alpha\in\Phi^-} |\alpha(\sse)|}.$$
The validity of Equation (\ref{eq:contraction}) follows.
\end{proof}

Recall that  $\Phi_{0}^+$ is the subset of the simple positive roots. Every positive root $\alpha \in \Phi^+$ can be written as an integral linear combination $\alpha = \sum_{\alpha_0 \in \Phi_{0}^+} n_{\alpha_0} \alpha_0$ of simple roots in a unique way. The height  $\hgt{\alpha}$  is the sum of the coefficients $n_{\alpha_0}$ in the above expression. Denote
\begin{equation}
\label{eq:height}
\hgt{\mathfrak{g}} = \sup_{\alpha \in \Phi^+} \hgt{ \alpha}.
\end{equation}

\begin{prop}
\label{prop:well understood ray}
There is a semisimple element $\sse_0 \in T$ and some $\lambda_0 > 1$ such that 
\begin{equation}
\label{eq:norms of adjoints}
\|\mathrm{Ad}(\sse_0^n)\| = \lambda_0^{n \hgt{ \mathfrak{g}}} 
\quad \text{and} \quad \|\Ad{\sse_0^{-n}}_{|\mathfrak{u}^-} \| = \lambda_0^{-n}
\end{equation} 
for all $ n \in \mathbb{N}$.
\end{prop}

\begin{proof}
For each simple relative root  $\alpha \in \Phi_{0}^+$ there is a one-parameter subgroup $\chi_\alpha$ of the torus $T$ so that
$$ \left<\alpha, \chi_\beta\right> =   N \delta_{\alpha\beta}$$
for some fixed $N \in \NN$ and all simple relative $k$-roots $\beta \in \Phi_0^+$ \cite[Proposition 8.6]{borel2012linear}. Fix any element $x_0 \in k$ with $|x_0|<1$.
Consider  the semisimple element $\sse_0 \in T$ given by
 $$\sse_0 = \prod_{\alpha \in \Phi_0^+} \chi_\alpha(x_0). $$
Note that $\alpha(\sse_0^n) = |x_0|^{nN \hgt{\alpha}}< 1$ for every positive root $\alpha \in \Phi^+$ and all $n \in \NN$. The result follows relying on   Proposition \ref{prop:the norm and inverse of norm of a semisimple element} and taking $\lambda_0 = |x_0|^{-N}$.
\end{proof}



\subsection*{Norms of restrictions to    nilpotent subalgebras} Consider the Lie algebra $\mathfrak{g}$ with its direct sum decomposition  $\mathfrak{g} = \mathfrak{u}^- \oplus \mathfrak{b}  $. 
The current setup fits into the framework of our discussion in \S\ref{sec:grassmanians}, where
\begin{itemize}
\item the Lie algebra $\mathfrak{g}$ plays the role of the vector space $V$,
	\item the    subgroup $\mathrm{Ad}(K)$ plays the role of the compact  subgroup of $\mathrm{GL}(V)$,  
	\item the two Lie subalgebras $\mathfrak{u}^-$ and $\mathfrak{b}$ respectively play the roles of the two direct summands $U$ and $U'$,
	\item   any fixed nilpotent subalgebra $\mathfrak{n}$ plays the role of the fixed subspace $W$, and
	\item  the linear operator $\mathrm{Ad}(\sse)$ for any fixed semisimple element $\sse \in T$ plays the role of the linear isomorphism $L$ preserving both direct summands $U$ and $U'$ as in Lemma 	\ref{lem:bound on norm inverse in terms of kappa}
.
\end{itemize}

In light of this correspondence, Theorem \ref{thm:goodness: probability of q being small} can be reformulated to obtain the following result. 
Consider the subset $\mathfrak{N}\subset \mathfrak{Gr}(\mathfrak{g})$ consisting of all  
\begin{itemize}
\item  the non-zero nilpotent Lie subalgebras $\mathfrak{n} \in \mathfrak{N}(\mathfrak{g})$ if $k = \RR$, or
\item  the Lie subalgebras $\Ad{g}\mathfrak{u}^+$ for some $g \in K$ if $k$ is non-Archimedean.
\end{itemize}

The collection $\mathfrak{N}$ of Lie subalgebras is closed in the Grassmannian  $\mathfrak{Gr}(\mathfrak{g})$. This    follows from   the preceding discussion  in the real case and from the compactness of  the group $K$    in the non-Archimedean case. 
 Let $\mathcal{F}(\mathfrak{N})$ be the family on $k$-analytic functions from the group $K$ associated to the collection $\mathfrak{N}$ as considered in Equations (\ref{eq:Qx}) and (\ref{eq:FN family}). Namely
  \begin{equation}
\mathcal{F}(\mathfrak{N}) = \{ g \mapsto \bigwedge_{i=1}^l \mathrm{P}g w_i \: : \: 
 \text{$\|w_i \| \le 1$ and $\mathrm{span}_k\{w_1,\ldots,w_l\} \in  \mathfrak{N}$} \}
 \end{equation}
 where $\mathrm{P} : \mathfrak{g} \to \mathfrak{u}^-$ is some fixed projection.

\begin{prop}
\label{prop:existence of expanding element}
There is a constant $\zeta > 0$ with the following property --- for all sufficiently large $\lambda > 1$ there exists a semisimple  element $\sse_\lambda \in T$ such that 
\begin{equation}
\label{eq:holds with one half}
   \eta_K \left(
\{g \in K \: : \: \inf_{0 \neq X \in \Ad{g} \mathfrak{n}} \frac{\| \Ad{\sse_\lambda} X \|}{\|X\|}
  \ge 2  \}
\right)
 \ge   1-\zeta \lambda^{- \frac{1}{ \dim_k K \mathrm{ord}_K \mathcal{F}(\mathfrak{N})}}
\end{equation}
and
 \begin{equation}
 \label{eq:second requirement}
 \|\Ad{\sse_\lambda}  \|   \le   \lambda^{\hgt{\mathfrak{g}}}
  \end{equation}
for all Lie subalgebras $\mathfrak{n}\in \mathfrak{N}$.

 \end{prop}

\begin{proof}

Consider the semisimple element $\sse_0 \in T$ and the constant $\lambda_0 > 1$ given in   Proposition \ref{prop:well understood ray}. We have
\begin{equation}
\label{eq:norms of adjoints '}
\|\mathrm{Ad}(\sse_0^n)\| = \lambda_0^{n \hgt{ \mathfrak{g}}} 
\quad \text{and} \quad \|\Ad{\sse_0^{-n}}_{|\mathfrak{u}^-} \| = \lambda_0^{-n}
\end{equation} 
for all $ n \in \mathbb{N}$.

Given any sufficiently large fixed $\lambda > 1$   there is an integer $n_0 \in \NN$ such that $\lambda_0^{n_0} \le \lambda \le \lambda_0^{n_0+1}$. Take $\sse_\lambda = \sse_0^{n_0}$. Observe that Equation (\ref{eq:second requirement}) holds with respect to the element $\sse_\lambda$ as
 $\|\Ad{\sse_\lambda}  \| = \lambda_0^{n_0 \hgt{\mathfrak{g}}} \le \lambda^{\hgt{\mathfrak{g}}}$.

We are now  in a position to apply our results   from the previous \S\ref{sec:grassmanians}.
According to Theorem  \ref{thm:goodness: probability of q being small} combined with Lemma 	\ref{lem:bound on norm inverse in terms of kappa}  
 there exists a constant $\kappa_\mathfrak{N} > 0$ such that every    Lie subalgebra $\mathfrak{n} \in \mathfrak{N}$ and every constant $\varepsilon > 0$ satisfy
 \begin{multline}
 \label{eq:end of proof of cor}
  \eta_K(\{g \in K \: : \: 
  \inf_{0 \neq X \in \Ad{g} \mathfrak{n}} \frac{\| \Ad{\sse_\lambda} X \|}{\|X\|} \ge \varepsilon \inf_{0 \neq X \in  \mathfrak{u}^-} \frac{\| \Ad{\sse_\lambda} X \|}{\|X\|}  \}) \ge \\
  \ge 1-  \kappa_\mathfrak{N}  \varepsilon^{\frac{1}{\mathrm{dim}_k K \mathrm{ord}_K \mathcal{F}(\mathfrak{N})}}.  
  \end{multline}
Note that Equation (\ref{eq:norms of adjoints '}) implies 
\begin{equation}
 \label{eq:end of proof of cor 2}
\varepsilon \inf_{0 \neq X \in  \mathfrak{u}^-} \frac{\| \Ad{\sse_\lambda} X \|}{\|X\|} \ge \varepsilon \lambda_0^n \ge  \varepsilon \frac{ \lambda}{\lambda_0}.
\end{equation}
 Take  $\varepsilon = \frac{2 \lambda_0 }{\lambda}$. 
 We conclude from the two Equations (\ref{eq:end of proof of cor}) and (\ref{eq:end of proof of cor 2}) that the desired Equation (\ref{eq:holds with one half}) holds true  with   the constant  
$\zeta = \kappa_\mathfrak{N} (2 \lambda_0)^{\frac{1}{\mathrm{dim}_k K \mathrm{ord}_K \mathcal{F}(\mathfrak{N})}}$.
  \end{proof}

\subsection*{The   parameters $\lambda(G), \delta(G)$ and $\sse(G)$.}
We now fix some parameters for the remainder of this work.
 Let the  constant $\zeta$ be as provided by Proposition \ref{prop:existence of expanding element}.
According to  Calculation \ref{calc:existence of good delta_0}   there is a sufficiently large value
\begin{equation}
\label{eq:fix lambda}
\lambda = \lambda(G) > 0
\end{equation}
such that the parameters
\begin{equation}
\label{eq:fixed a's}
 a_1 =  2,  \quad a_2(G) =    \lambda(G)^{-\hgt{\mathfrak{g}}}   \quad \text{and} \quad p(G)  = p_\lambda = 1-\zeta \lambda^{- \frac{1}{ \dim_k K \mathrm{ord}_K \mathcal{F}(\mathfrak{N})}}
 \end{equation}
 satisfy Equation  (\ref{eq:balance}) and such that the corresponding constant 
 \begin{equation}
\label{eq:fix delta}
 \delta = \delta(G) = \delta_0(a_1, a_{2,\lambda}; p_\lambda) \end{equation}
is bounded from below as follows\footnote{The constant $2$ appearing in the first line of Equation (\ref{eq:sufficiently large delta}) can be replaced by any real number in the large $(1,\infty)$. The integral value of $2$ is only a matter of convenience.}
\begin{equation}
\begin{aligned}
 \label{eq:sufficiently large delta}
\delta^{-1}    &\le     2  \hgt{\mathfrak{g}}  \dim_k K  \mathrm{ord}_K \mathcal{F}(\mathfrak{N}) \le \\
&\le 
 \begin{cases} 
 2 \hgt{\mathfrak{g}}  \dim_\RR K (6  \hgt{\mathfrak{g}}  \dim_\RR U+1)^{\mathrm{rank}(K)}
  & \text{$k$ is $\RR$,} \\
     8  \hgt{\mathfrak{g}}^2  \dim_k K  \dim_k G  \dim_k^3 U 
     &
      \text{$k$ is non-Arch.}
    \end{cases} 
    \le \\
&\le     \begin{cases} 
     (3 \hgt{\mathfrak{g}} \dim_\RR G)^{(\mathrm{rank}(K)+1)}     & \text{$k$ is $\RR$,} \\
       \hgt{\mathfrak{g}}^2  \dim^5_k G & \text{$k$ is non-Arch.}
    \end{cases} 
 \end{aligned}
\end{equation}
The upper bounds on  $\mathrm{ord}_K \mathcal{F}(\mathfrak{N})$   follow from  Propositions \ref{prop:order of adjoint in real case} and \ref{prop:order in pos char case} respectively.
 Lastly, for the remainder of this work fix  the semisimple element
\begin{equation}
\label{eq:the semisimple element a}
\sse = \sse(G) = \sse_{\lambda } \in T
\end{equation}
such that the two Equations (\ref{eq:holds with one half}) and (\ref{eq:second requirement}) of Proposition \ref{prop:existence of expanding element}   
are satisfied. 


\section{Subgroups generated by small elements}
\label{sec:Zass}

We continue using the notations introduced in \S\ref{sec:nilpotent} so that  in particular $G$ is a semisimple $k$-analytic group over the local field $k$.
The goal of this section is to study discrete subgroups of $G$ generated by small elements.  We  introduce a function $\mathcal{I}_G$ defined on the Chabauty space  of discrete subgroups $\Subd{G}$ which, roughly speaking, assigns to every discrete subgroup $\Gamma$ of $G$ the norm of its smallest non-trivial element.

Let $\mathfrak{g} = \mathrm{Lie}(G)$ be the Lie algebra  of the $k$-analytic group $G$. We  consider $\mathfrak{g}$ with some $\Ad{K}$-invariant norm $\|\cdot\|$ as   in \S\ref{sec:nilpotent}. 
It will be useful to introduce the following notation
$$\mathrm{B}_\mathfrak{g}(r) = \{ X \in \mathfrak{g} \: : \: \|X\| \le r \} $$
for every  radius $r > 0$.

\subsection*{The real case.  Zassenhaus  lemma.}

Assume that the local field $k$ is $\RR$ so that the group $G$ is a semisimple real Lie group without compact factors. 

The exponential map $\exp$ sets up a local diffeomorphism from $ 0\in \mathfrak{g}$ to $\mathrm{id}_G \in G$. We will use $\log$ to denote the inverse of the exponential map $\exp$ where defined.

We  rely on the classical  lemma due to Zassenhaus \cite{zassenhaus1937beweis}. It was proved also in \cite[Lemma 2]{kazhdan1968proof}. Roughly speaking, it says that a discrete subgroup generated by small elements of $G$ is nilpotent. The  precise formulation is as follows.

\begin{lemma}[Zassenhaus]
\label{lem:Zassenhaus--Kazhdan--Margulis}
There is a radius $ R = R(G) > 0$ and an identity neighborhood 
$V = \exp \mathrm{B}_\mathfrak{g}(R)$ such that for every discrete subgroup $\Gamma$ of $G$ the group generated by $\Gamma \cap V$ is contained in a connected  nilpotent Lie subgroup.
\end{lemma}
%

For every discrete subgroup $\Gamma$ of $ G$ let   $N(\Gamma)$ denote the minimal connected nilpotent Lie subgroup of $G$ containing $\left<\Gamma \cap V\right>$. 
Denote $\mathfrak{n}(\Gamma) = \mathrm{Lie}(N(\Gamma))$.

\subsection*{The non-Archimedean case. Springer isomorphism.}

Assume that $k$ is a non-Archimedean local field of positive characteristic   and that $\mathrm{char}(k)$ is a good prime for the semisimple group $\GG$.  

Let $\mathcal{O}$ be the ring of integers of the non-Archimedean local field $k$. Let $\mathfrak{m}$ be the maximal ideal of $\mathcal{O}$.
 Let $ \GG(\mathfrak{m}^i)$ denote the congruence subgroup of the compact group $\GG(\mathcal{O})$ with resect to the ideal $\mathfrak{m}^i$ for every $i \in \NN$. 
 
We define a norm  $\|\cdot\|_\mathfrak{m}$ on the group $\GG(\mathfrak{m})$ by letting
\begin{equation}
\label{eq:p norm}
 \|g\|_\mathfrak{m} = 
 \inf  \; \{  \left| \sfrac{\mathcal{O}}{\mathfrak{m}}\right|^{- i}   \: : \: \text{$i \in \NN$ and $g \in G(\mathfrak{m}^i)$}    \} 
 \end{equation}
for all elements $g  \in \GG(\mathfrak{m})$. 

It is well-known that every discrete subgroup of $\GG(\mathfrak{m})$  is   finite and unipotent. Indeed the subgroup $\GG(\mathfrak{m})$ is a   pro-$p$-group \cite[Lemma \S3.3.8]{platonov1992algebraic}. Therefore every discrete subgroup $\Gamma$ of $\GG(\mathfrak{m})$ is  a finite  $p$-group and is in particular nilpotent. We obtain
 $$(g - \mathrm{id}_G)^{p^l} = g^{p^l} - \mathrm{id}_G^{p^l} = g^{p^l} - \mathrm{id}_G = 0$$ 
 for some $l \in \NN$ and all $g \in \Gamma $. Therefore $g - \mathrm{id}_G$ is nilpotent and $g$ is unipotent.



\vspace{5pt}

Let $\mathfrak {U} $ be the subset of $\GG$ consisting of all unipotent elements. Then  $\mathfrak{U}$ is an irreducible closed subvariety of $\GG$ defined over $k$ and   $\dim \mathfrak{U} = \dim \GG -   \mathrm{rank} \GG$  \cite[Proposition 1.1]{springer1969unipotent}.
Likewise, the subset $\mathfrak{B} \subset \mathrm{Lie}(\GG)$ of all nilpotent elements  is   an irreducible closed subvariety of $\mathrm{Lie}(\GG)$ defined over $k$   \cite[Proposition 2.1]{springer1969unipotent}.

\begin{theorem}[Springer isomorphism]
\label{thm:springer isomorphism}
Assume that the characteristic of $k$ is a good prime for $\GG$. Then there exists a $\GG$-equivariant $k$-isomorphism $\Sigma : \mathfrak{U} \to \mathfrak{B}$. Moreover  $\Sigma(U) = \mathfrak{u}^+$.
\end{theorem}
Here $U$ is the group of $k$-rational points of the unipotent radical of a minimal parabolic $k$-subgroup.  The Lie subalgebra $\mathfrak{u}^+$ is discussed in \S\ref{sec:nilpotent}.
\begin{proof}[Proof of Theorem \ref{thm:springer isomorphism}]
Springer constructed in  \cite[Theorem 3.1]{springer1969unipotent} a $\GG$-equivariant $k$-morphism $\Sigma : \mathfrak{U} \to \mathfrak{B}$ such that $\Sigma : \mathfrak{U}(\overline{k}) \to \mathfrak{B}(\overline{k})$ is a homeomorphism. This statement was strengthened in 
  \cite[Corollary 9.3.4]{bardsley1985etale} where $\Sigma$ is shown to be a $k$-isomorphism. The fact that $\Sigma(U) = \mathfrak{u}^+$ follows from the proof.
\end{proof}

Note that $\Sigma (\mathfrak{U}(k))  = \mathfrak{B}(k)$. The restriction of $\Sigma$ to the compact Hausdorff  space $\mathfrak{U}(\mathfrak{m})$ is a $k$-morphism which is a homeomorphism onto its image. Therefore there is a constant $\sigma > 1$ such that
 \begin{equation}
\label{eq:sigma}
  \sigma^{-1} \|u \|_\mathfrak{m} \le \|\Sigma( u)\|  \le \sigma \|u\|_\mathfrak{m} 
 \end{equation}
for every unipotent element $u \in \mathfrak{U}(\mathfrak{m})$.

\vspace{5pt}

\emph{ Abusing notation, we will from now on use $\log$ and $\exp$ to denote the Springer isomorphism $\Sigma$ and its inverse where defined, even through   $\Sigma$ is not a logarithmic map in the strict sense.}

%

\subsection*{The  function $\mathcal{I}_G$}

In the Archimedean case let $R > 0$ be the radius given by the Zassenhaus lemma (i.e. Lemma \ref{lem:Zassenhaus--Kazhdan--Margulis}). In the positive characteristic   case   let $R =  \sigma^{-1} \left| \sfrac{\mathcal{O}}{\mathfrak{m}}\right|^{-1}$ where the constant $\sigma$ is as given in Equation (\ref{eq:sigma}). Denote 
\begin{equation}V = \exp \mathrm{B}_\mathfrak{g}(R).
\end{equation}
 Namely $V$ is a Zassenhaus neighborhood in the Archimedean case and is an identity neighborhood contained in the subgroup $G(\mathfrak{m})$ in the non-Archimedean case.

  Let $0 < \rho < R$ be a sufficiently small radius such that the identity neighborhood 
  \begin{equation}V_0 = \exp \mathrm{B}_\mathfrak{g}(\rho)\end{equation}
  satisfies
  \begin{equation} V_0 \subset V \cap V^{\sse} \cap V^{\sse^{-1}}\end{equation}
where $\sse = \sse (G) \in T$ is the particular  semisimple element fixed in the final paragraph of \S\ref{sec:nilpotent}, see Equation (\ref{eq:the semisimple element a}).

 
We introduce a real-valued \emph{discreteness radius} function $\mathcal{I}_G$   on the Chabauty space $\Subd{G}$ of   discrete subgroups of the group $G$. Given a discrete subgroup $\Gamma $ of $G$ define 
\begin{align}
\label{eq:f_G}
\begin{split}
\mathcal{I}_G(\Gamma)  &= 
\begin{cases}
\min_{\gamma \in (\Gamma \cap V_0 )\setminus\{e\}} \|\log \gamma \|& \Gamma \cap V_0  \neq \{e\}\\
\rho&\Gamma \cap V_0  = \{e\}
\end{cases} \\
 &= \sup \{ 0 < r < \rho \: : \: \Gamma \cap \exp \left( \mathrm{B}_\mathfrak{g}(r)\right)  = \{ e \} \}   
\end{split}
 \end{align}
The function $\mathcal{I}_G$ takes values in the interval $\left(0,\rho\right]$. 
In the non-Archimedean case $\log$ denotes the Springer isomorphism $\Sigma$ and   Equation (\ref{eq:f_G}) depends on the fact that $\Gamma \cap V_0$ is unipotent combined with Theorems \ref{thm:Gille} and \ref{thm:springer isomorphism}.  

The logarithm map (and the Springer isomorphism) are $G$-equivariant, namely
\begin{equation} \label{eq:equivariance}
\log(x^g) = \Ad{g}\log(x) \end{equation}
for every element  $x \in V \cap V^{g^{-1}}$. 
As the norm $\|\cdot\|$ on the Lie algebra $\mathfrak{g}$ is   $\mathrm{Ad}(K)$-invariant, the function $\mathcal{I}_G$ is   $K$-invariant with respect to the     compact subgroup $K$,  namely
\begin{equation} 
\label{eq:Ad K changing f}
 \mathcal{I}_G(\Gamma^k) =  \mathcal{I}_G(\Gamma) \quad \forall k \in K. \end{equation}

\subsection*{Contraction on the average} Let $\Gamma \in  \Subd{G}$ be a discrete subgroup of $G$. We show  that   conjugation by the semisimple   element $\sse$ expands the  function $\mathcal{I}_G$ by a definite amount, when it is   evaluated at \emph{most} conjugates of $\Gamma$ by an element  of  the   compact subgroup $K$.

Let $\eta_K$ denote the Haar measure of the compact group $K$ normalized to be a probability measure so that $\eta_K(K) = 1$. The precise statement is as follows.

\begin{prop}
 \label{prop:prob integral expands}
 If the discrete subgroup $\Gamma$ 
 satisfies
   $\mathcal{I}_G(\Gamma) \le \frac{\rho}{2}$ then 
\begin{equation}
\label{eq:f expands on the average}
\eta_K(\{g \in K \: : \: \mathcal{I}_G(  \Gamma ^{g\sse} ) \ge   2   \mathcal{I}_G (\Gamma^g) \}) \ge p(G)
\end{equation} 
where the parameter $p(G)   >0$ is as given in Equation (\ref{eq:fixed a's}).
\end{prop}
 
 We postpone the proof of Proposition \ref{prop:prob integral expands}  until the end of the current \S\ref{sec:Zass}.


 

\begin{prop}
\label{prop:global expansion}
For every discrete group $\Gamma$ we have
\begin{equation} \mathcal{I}_G(\Gamma^{\sse}) \ge \frac{\mathcal{I}_G(\Gamma)}{\| \mathrm{Ad}(\sse^{-1} )\| }.\end{equation}
\end{prop}
\begin{proof}
As $\|\mathrm{Ad}(\sse^{-1})\| > 1$ and $0 <  \mathcal{I}_G(\Gamma) \le \rho$   the result clearly  holds if $\mathcal{I}_G(\Gamma^{\sse}) = \rho$. Assume therefore   that $\mathcal{I}_G(\Gamma^{\sse}) < \rho$. There is  a  non-trivial element $\gamma \in \Gamma$ with $\gamma^{\sse} \in V_0$ and $\mathcal{I}_G(\Gamma^{\sse}) = \|\log \gamma^{\sse} \|$. Since $V_0^{\sse^{-1}} \subset V$ we have that $\gamma \in V$.  The equivariance of the logarithm map (or   the Springer isomorphism)   formulated in Equation (\ref{eq:equivariance}) implies that 
\begin{equation}\mathcal{I}_G(\Gamma^{\sse}) = \| \log \gamma^{\sse} \| = \frac{\| \log \gamma\|}{\| \mathrm{Ad}(\sse^{-1}) \| } \ge \frac{\mathcal{I}_G(\Gamma)}{\| \mathrm{Ad}(\sse^{-1}) \| }\end{equation}
as required.
\end{proof}

 

\begin{prop}
\label{prop:expansion of single subgroup}
Let $\mathfrak{h} \le \mathfrak{g}$ be any Lie subalgebra with $\log(\Gamma \cap V) \subset \mathfrak{h}$. 
If   the discrete group $\Gamma$ and the element $\sse \in T$ satisfy
\begin{equation}
\label{eq:assumption}
\inf _{0 \neq X \in \mathfrak{h}} \frac{\| \Ad{s} X \|}{\|X \| } \ge 2
\quad \text{and} \quad \mathcal{I}_G(\Gamma) \le \frac{\rho}{2}  \end{equation}
then
\begin{equation}
\label{eq:conclusion}
 \mathcal{I}_G(  \Gamma ^{\sse} ) \ge   2  \mathcal{I}_G (\Gamma).
 \end{equation}
\end{prop}
\begin{proof}
Assume that two conditions given in Equation (\ref{eq:assumption}) are satisfied. 
This implies $\mathcal{I}_G(\Gamma) < \rho$ so that    $\Gamma \cap V_0 \neq \{e\}$. 
Observe  that 
\begin{equation}
\label{eq:Ad a expands}
\| \log \gamma^{\sse}  \| = \| \Ad{\sse}\log \gamma  \| \ge  2 \|\log \gamma\| 
 \end{equation}
for every element $  \gamma \in \Gamma \cap V \cap V^{\sse^{-1}}$. In particular, as $V_0 \subset V \cap V^{\sse^{-1}}$ it follows that Equation (\ref{eq:Ad a expands}) holds for every $\gamma \in \Gamma \cap V_0$. On the other hand, note that
\begin{equation}
\label{eq:intersection and then conjugation}
\Gamma^{\sse} \cap V_0 = (\Gamma \cap V \cap V^{\sse^{-1}})^\sse \cap V_0.
\end{equation}
The containment in the  non-trivial direction of Equation (\ref{eq:intersection and then conjugation}) holds true since $V_0 \subset V \cap V^{\sse}$. 
The required Equation (\ref{eq:conclusion}) follows by combining Equations (\ref{eq:assumption}), (\ref{eq:Ad a expands}) and (\ref{eq:intersection and then conjugation}) with the  definition of the function $\mathcal{I}_G$.
\end{proof}

We are   ready to complete the proof of the main result of the current \S\ref{sec:Zass}.

\begin{proof}[Proof of Proposition \ref{prop:prob integral expands}]
Assume that the   subgroup $\Gamma$ satisfies $\mathcal{I}_G(\Gamma) \le \frac{\rho}{2}$. In particular  $\mathcal{I}_G(\Gamma^g) \le \frac{\rho}{2 }$ for all elements $ g \in K$, see Equation (\ref{eq:Ad K changing f}). 

In the Archimedean case   take $N = N(\Gamma)$  relying on the Zassenhaus Lemma and let $\mathfrak{n} = \mathfrak{n}(\Gamma) = \mathrm{Lie}(N)$ be the corresponding Lie algebra. 
In the non-Archimedean case  let $U$   be the unipotent radical of  \emph{some} parabolic subgroup containing $\Gamma \cap V$ relying on  Theorem \ref{thm:Gille} and let $\mathfrak{n} = \mathfrak{u}^+$ be the corresponding Lie subalgebra. In both cases $\log(\Gamma \cap V) \subset \mathfrak{n}$,  see Theorem \ref{thm:springer isomorphism} for the non-Archimedean case


Let $p(G) > 0$ be the parameter fixed in the last paragraph of \S\ref{sec:nilpotent}.  We deduce from Proposition   \ref{prop:existence of expanding element} that 
\begin{equation}
\label{eq:probs probs}
\eta_K \left(
\{g \in K \: : \: \inf_{0 \neq X \in \Ad{g} \mathfrak{n}} \frac{\| \Ad{\sse} X \|}{\|X\|}
  \ge 2    \}
\right)
\ge  p(G). 
\end{equation}
The desired 
  Equation (\ref{eq:f expands on the average}) follows by combining the above Equation (\ref{eq:probs probs}) with Proposition \ref{prop:expansion of single subgroup}.
\end{proof}

\section{The key inequality and its applications}
\label{sec:key-inequality}

We maintain   the notations introduced in the previous \S\ref{sec:nilpotent} and \S\ref{sec:Zass} so that   $G$ is a semisimple  analytic group,   $K$ is a   compact subgroup with  Haar probability measure $\eta_K$ and    $ \sse  $ is the  particular   semisimple element fixed    in the last paragraph of  \S\ref{sec:nilpotent}.  

We will work with the  Borel probability  $
\mu_{\sse}$ on the   group $G$ given by
\begin{equation}
\label{def:nu_a}
\mu_{\sse} = \eta_K * \delta_{\sse} * \eta_K.
\end{equation}

The main goal of the current \S\ref{sec:key-inequality}  is to prove the key inequality (Theorem \ref{thm:inequality}) and derive some of its applications. For the reader's convenience we    restate the inequality   using  the   convolution operator $A_{\mu_{\sse}}$ introduced in Equation (\ref{eq:convolution operator}).


 
\begin{theorem*}[The key inequality]
\label{thm:key inequality restated}
Consider the discreteness radius function $\mathcal{I}_G : \Subd{G} \to \left(0,\rho\right]$ given by
\begin{equation}
  \mathcal{I}_G(\Gamma)= \sup \: \{ 0 <  r  < \rho \: : \: \Gamma  \cap \exp\left( \mathrm{B}_\mathfrak{g} (r) \right) = \{\mathrm{id}_G \} \} 
\end{equation}
for any discrete subgroup $\Gamma$ of $G$.  There are constants $0 < c <1$  and $b  > 0$   such that the function $\mathcal{I}_G$  satisfies 
\begin{equation}
\label{eq:convolution with I_G}
A_{\mu_{\sse}}   \mathcal{I}_G^{-\delta }   \le c   \mathcal{I}_G^{-\delta } + b
 \end{equation}
 where the parameter $\delta = \delta(G)$ is as    in  Equation        (\ref{eq:fix delta}).  
 \end{theorem*}
 
\begin{proof}

 %
 
 Let the constants  $a_1 = 2, 0 < a_2(G) < 1 $ and $0 < p(G) < 1 $ be as  in  Equation (\ref{eq:fixed a's}). Moreover denote $\rho_0 = \frac{\rho}{2 }$. 
 Observe that every discrete subgroup $\Gamma \in \Subd{G}$   satisfies
\begin{enumerate}
\item if $\mathcal{I}_G(\Gamma) \le  \rho_0$ then 
$\mu_{\sse}( \{g \in G \: : \: \mathcal{I}_G(\Gamma^g) \ge a_1 \mathcal{I}_G(\Gamma)  \}) \ge p(G)$ and  
\item $\mathcal{I}_G(\Gamma^g) \ge a_2(G) \mathcal{I}_G(\Gamma)$ holds   every element $g \in K\sse K = \mathrm{supp}(\mu_{\sse})$.
\end{enumerate}
 Statements (1) and (2) follow  from  Propositions \ref{prop:prob integral expands} and \ref{prop:global expansion} respectively, combined with Equation (\ref{eq:Ad K changing f}) to take into account conjugation by elements of   the      compact subgroup $K$.    
Putting all this together and applying Proposition \ref{prop:from expansion to contraction}  we  immediately deduce the desired Equation     (\ref{eq:convolution with I_G}).
\end{proof}

\subsection*{Effective weak uniform discreteness}
\label{sub:QWUD}

The two main  Theorems \ref{thm:main theorem} and \ref{thm:main theorem in positive char} appearing in the introduction   follow immediately   from   Theorem \ref{thm:main theorem - local fields}, which we are now in a position to state and prove. 

Let $\mu$ be any bi-$K$-invariant probability measure on the group $G$ whose support $\mathrm{supp}(G)$ generates the group. 

\begin{theorem}
\label{thm:main theorem - local fields}
There are   constants $\beta,  \delta > 0$     such that every   $\mu$-stationary probability $G$-space $(Z,\nu)$   with $\nu$-almost everywhere discrete stabilizers satisfies 
\begin{equation}
\label{eq:main equation - local fields}
 \nu (\{z \in Z \: : \: \mathcal{I}_G(G_z) < \varepsilon \})      <  \beta \varepsilon^\delta
 \end{equation}
for all $\varepsilon > 0$. The constant $\delta  $ satisfies a lower bound as in Equation (\ref{eq:sufficiently large delta}).
\end{theorem}

Prior to proving  Theorem \ref{thm:main theorem - local fields} we observe that any $\mu$-stationary $G$-space $(Z,\nu)$ is also $\mu'$-stationary for any other probability measure $\mu'$ as above.  In other words, the statements of    Theorems \ref{thm:main theorem}, \ref{thm:main theorem in positive char} and \ref{thm:main theorem - local fields}  are independent of the   choice of   $\mu$.

 Indeed, the Poisson boundary of the pair $(G,\mu)$ can be identified with the homogoenous space $(G/B,\eta_B)$ where $B$ is a minimal $k$-parabolic subgroup and $\nu_B$ is the unique $K$-invariant probability measure\footnote{The uniqueness of   $\eta_B$ follows from the Iwasawa decomposition, i.e. the transitivity of the $K$-action on the homogenous space $G/B$.} on $G/B$. 
From the universal property of the Poisson boundary, a measure $\nu$ on a $G$-space $Z$ is $\mu$-stationary if and only if $\nu$ is the $\eta_B$-barycentre of some measurable map from $G/B$ to the space $\text{Prob}(Z)$ of probability measures on $Z$. This condition depends only on $\eta_B$ and not on the specific choice of the measure $\mu$.
%

%

\begin{proof}[Proof of Theorem \ref{thm:main theorem - local fields}]
Let $(Z,\nu)$ be any $\mu$-stationary probability $G$-space with $\nu$-almost surely discrete stabilizers. By the preceding discussion we may assume that $\nu$ is   $\mu_{\sse}$-stationary where $\mu_{\sse} = \eta_K * \delta_s * \eta_K$ as in Equation (\ref{def:nu_a}).

Consider the probability measure $\nu_*$ on the Chabauty space of discrete subgroups $\Subd{G}$ obtained as the push-forward\footnote{Varadarajan's compact model theorem \cite[2.1.19]{zimmer2013ergodic} implies that the stabilizer map  is $\nu$-measurable.} of the measure $\nu$ via the stabilizer map $Z\to\text{Sub}(G),~z\mapsto G_z$.  As the stabilizer map is $G$-equivariant, the measure $\nu_*$ is $\mu$-stationary (as well as $\mu_{\sse}$-stationary).

To conclude the proof we apply  Lemma \ref{lem:bound on contracting function} with respect to the function $\mathcal{I}_G^{-\delta}$ on the Chabauty space of discrete subgroups $\Subd{G}$. Observe  that the discreteness radius function $\mathcal{I}_G$ is continuous on  $\Subd{G}$ with respect to the Chabauty topology \cite{delaharpe2008spaces}. The key equality    coincides with Equation (\ref{eq:contraction F topological}).
 We  obtain the  estimate
\begin{equation}
\label{eq:essentially same}
 \nu_*  (\{\Gamma \in \Subd{G} \: : \: \mathcal{I}_G^{-\delta}(\Gamma) \ge M\}) \le   \beta M^{-1} 
 \end{equation}
for all   $M > 0$   and with $\beta = \frac{b}{1-c}$.  Taking $\varepsilon = \frac{1}{M}$ and rearranging gives
\begin{equation}
\label{eq:essentially same rearrange 2}
 \nu_* (\{\Gamma \in \Subd{G} \: : \: \mathcal{I}_G(\Gamma)   \ge \varepsilon\}) \le   \beta \varepsilon^{\delta}
 \end{equation}
 for all $ \varepsilon > 0$. The desired conclusion follows since $\nu_*$ is the push-forward of $\nu$.
 \end{proof}


In the remainder of the current \S\ref{sec:key-inequality} we complete the proofs of   other results stated in the introduction.

\subsection*{The real Lie group case}

Let $G$ be a real semisimple   Lie group without compact factors. 
For the purpose of proving the real case of our main result (Theorem \ref{thm:main theorem}) we may replace the  Lie group $G$ by the group $G^0/(G^0 \cap Z(G))$. In other words, we may assume without loss of generality that the Lie group $G$ is connected and center-free. 
There exists a connected $\RR$-linear semisimple algebraic group $\GG$ such that $G = \GG(\RR)^0$. The algebraic group $\GG$ is in fact defined over $\QQ$ \cite[3.1.6]{zimmer2013ergodic}. The semisimple algebraic group $\GG$ has no $\RR$-anisotropic factors since  the Lie group $G$ has no compact factors. 

The norm on $B_\mathfrak{g}(R)$ given by $g \mapsto \| \log g\|$ is bi-Lipschitz equivalent to any fixed left-invariant Riemannian (or Finsler) metric on the Lie group $G$. 
However note that the value of the function $\mathcal{I}_G$ for a discrete subgroup $\Gamma$   can only be used to determine whether $\Gamma \cap \mathrm{B}_\varepsilon \neq \{e\}$   provided that $\varepsilon < \rho$ where $\rho$ is as in \S\ref{sec:Zass}.
 
We conclude that Theorem \ref{thm:main theorem} of the introduction follows immediately from the real case of Theorem  \ref{thm:main theorem - local fields}.

%

\subsection*{The positive characteristic case}

 Theorem \ref{thm:main theorem in positive char} of the introduction is essentially stated in terms of the norm $\|\cdot\|_\mathfrak{m}$ on the $k$-analytic group $G$   defined in Equation (\ref{eq:p norm}). The Springer isomorphism $\Sigma$ is bi-Lipschitz with constant $\sigma > 1$ according to  Equation $(\ref{eq:sigma})$. Therefore    Theorem \ref{thm:main theorem in positive char}  follows from the   non-Archimedean case of Theorem \ref{thm:main theorem - local fields}.
  Note   that the value of the function $\mathcal{I}_G$ for a discrete subgroup $\Gamma$  can only be used to determine whether $\Gamma \cap \GG(\mathfrak{m}^i)  \neq \{e\}$   provided that $\left| \sfrac{\mathcal{O}}{\mathfrak{m}}\right|^{- i}  < \rho$. 

\subsection*{Discrete $\mu$-stationary random subgroups}

Let $\mu$ be a  probability measure on the group $G$ such that $\mathrm{supp}(\mu)$ generates $G$. Let $\mathrm{DRS}_{\mu} (G)$ denote the space of all discrete $\mu$-stationary random subgroups of the group $G$, i.e. all $\mu$-stationary probability measures $\nu$ on $\Sub{G}$ satisfying $\nu(\Subd{G}) = 1$. We show that the space  $\mathrm{DRS}_{\mu} (G)$  is compact in the   weak-$*$ topology  of probability measures.

\begin{proof}[Proof of Corollary \ref{cor:compact}]
 Any weak-$*$ limit of $\mu$-stationary random subgroups  (i.e. probability measures on $\Sub{G})$  is  also a $\mu$-stationary random subgroup. Assume towards contradiction that  $\nu_n$ is a sequence of discrete $\mu$-stationary random subgroups satisfying $\nu_n \to \nu_0$ where $\nu_0(\Subd{G}) < 1$. In other words
\begin{equation}
\label{eq:lim inf}
 \liminf_{ r \to 0} \nu_0( \{ H \in \Sub{G} \: : \: H \cap \mathrm{B}_r \neq \{\mathrm{id}_G\} \}) > 0.
 \end{equation}
Since the condition of intersecting $\mathrm{B}_r$  non-trivially   is Chabauty open, it follows from the Portmanteau theorem  that Equation (\ref{eq:lim inf}) contradicts Theorem \ref{thm:main theorem - local fields}.
\end{proof}

\subsection*{Evanescence for locally symmetric manifolds}

Recall that a Riemannian manifold of non-positive sectional curvature   is called \emph{evanescent} if its injectivity  radius function vanishes at infinity. We  now prove that evanescence implies finite volume for locally symmetric spaces.

\begin{proof}[Proof of Theorem \ref{thm:evanescence}]
Let $M = K  \backslash G / \Gamma$ be a locally symmetric space for some semisimple   Lie group $G$ with maximal compact subgroup $K$ and a discrete torsion-free subgroup $\Gamma$ of $G$. The homogenous space $  G / \Gamma$ admits a natural quotient map $$ \pi : G/\Gamma\to M, \quad \pi : g\Gamma \to Kg \Gamma \in M.$$ 
Furthermore the natural $G$-invariant measure $\sigma$ on $G/ \Gamma$ pushes forward to give  a Riemannian volume $\pi_* \sigma$ on the manifold $M$. Note that the locally symmetric space $M$ is evanescent if and only if the continuous function $g\Gamma \to \mathcal{I}_G(g\Gamma g^{-1})^{-\delta}$ is proper on the homogenous space $G/\Gamma$. The desired conclusion follows from the key inequality (Theorem \ref{thm:inequality}) combined with the argument of \cite[Corollary 1.5]{margulis2004random}.
 \end{proof}

 \section{Weakly cocompact lattices}
 \label{sec:local spectral gap}

We briefly discuss the notion  weakly cocompact lattices in general. We then apply our main result Theorem   \ref{thm:main theorem in positive char} towards the study   of this notion.
 
\subsection*{Weakly cocompact lattices}

Let $H$ be a locally compact group and $\Gamma$ be a lattice in the group $H$.  The lattice $\Gamma$   is \emph{weakly co-compact} if the unitary quasiregular representation of the group $H$ on the  Hilbert space $L^2_0(H/\Gamma)$ admits no almost invariant vectors.

\begin{lemma}
\label{lem:local spectral gap at infinity implies spectral gap}
Let $\mu$ be a compactly supported probability measure on the group $H$.  
Let $D \subset H/\Gamma$ be a compact subset.  Let $\mathrm{P} : L^2(H/\Gamma) \to L^2( (H/\Gamma) \backslash D)$ be the orthogonal projection operator. If 
\begin{equation}
\label{eq:less than one}
\| \mathrm{P} A_\mu \mathrm{P} _{|L^2(H/\Gamma \backslash D)} \| < 1
\end{equation}
then the lattice $\Gamma$ is weakly cocompact.
\end{lemma}
\begin{proof}
Denote $\hat{D} = \mathrm{supp}(\mu) D$. Let $\hat{\mathrm{P}}$ be the orthogonal projection  operator
   $$\hat{\mathrm{P}} : L^2(H/\Gamma ) \to L^2((H/\Gamma)\backslash \hat{D}).$$
It is clear that   $\mathrm{P} A_\mu \hat{\mathrm{P}} = A_\mu \hat{\mathrm{P}}$ and that $\|\hat{\mathrm{P}}f\| \le \|\mathrm{P} f\|$ for every vector $f \in L^2(H/\Gamma)$.

Assume towards contradiction that the lattice $\Gamma$ is not weakly cocompact in the group $H$. 
According to \cite[Lemma 1.9, \S III.1]{margulis1991discrete} there exists a sequence $v_n \in L^2(H/F )$ of functions with $\|v_n\| = 1, \|A_\mu v_n\| \to 1$ and $d_n = \|\hat{\mathrm{P}} v_n -   v_n\| \to 0$. Since $A_\mu$ are $\mathrm{P}$ are both  contracting operators we obtain that
\begin{align*}
 \|A_\mu v_n \| & \le \|A_\mu \hat{\mathrm{P}} v_n \|  + \|A_\mu ( \hat{\mathrm{P}} v_n -   v_n)\| \le 
 \|\mathrm{P} A_\mu \hat{\mathrm{P}} v_n \|  + d_n \le \\
 &\le  \|\mathrm{P} A_\mu  \mathrm{P} v_n \|  + d_n   \le  \|\mathrm{P} A_\mu \mathrm{P}_{|L^2(H/\Gamma \backslash D)} \|   + d_n 
\end{align*}
for all $n \in \NN$. This is a contradiction to the assumption stated in Equation (\ref{eq:less than one}).
\end{proof}

%
\subsection*{The thin part}
Let us return to our main case of interest.   Namely, let $k$ be either the field $\RR$ or a non-Archimedean local field.  Let $\GG$ be a connected simply-connected  linear $k$-algebraic semisimple group with $k$-anisotropic factors. Denote $G = \GG(k)$ so that $G$ is a $k$-analytic group.

Let $\mu$ be any bi-$K$-invariant probability measure on the group $G$ whose support $\mathrm{supp}(\mu)$ generates the group. Let $\nu$ be any $\mu$-stationary probability measure on the Chabauty space $\Subd{G}$ of discrete subgroups.

The following elementary integrability result is an immediate consequence of our main result Theorem \ref{thm:main theorem - local fields}. Its proof is   general and can be applied to any function satisfying the conclusion of Lemma \ref{lem:bound on contracting function}.

 \begin{prop}
 \label{prop:f_G^delta is integrable}
 The function $\mathcal{I}_G^{-\frac{\delta}{2}}$ is $\nu$-integrable.
 \end{prop}

 \begin{proof}
Let $D : \left[0,1\right] \to \RR_{\ge 0}$ be the inverse   cumulative distribution function corresponding to $\mathcal{I}_G^{-\frac{\delta}{2}}$. It is defined as
\begin{equation} D(t) = \inf \{ x \in \RR_{\ge 0} \: : \: \nu(\{ \Gamma \in \Subd{G} \: : \: \mathcal{I}_G^{-\frac{\delta}{2}}(\Gamma) \le x\}) \ge t \}. \end{equation}
In other words, the function $D$ satisfies
\begin{equation} \nu( \{ \Gamma \in \Subd{G} \: : \: \mathcal{I}_G^{-\frac{\delta}{2}}(\Gamma) \le  D(t) \} ) = t \quad \forall t \in \left[0,1\right].\end{equation}
Theorem \ref{thm:main theorem - local fields}
 implies
\begin{equation}
 \nu (\{\Gamma \in \Subd{G} \: : \: \mathcal{I}^{-\frac{\delta}{2}}_G(\Gamma) \le M^{-\frac{1}{2}}\}) \ge  1-   \beta M
\end{equation} for some constant $\beta > 0$ and all sufficiently small $M > 0$.
It follows that\begin{equation} D(1 - \beta t) \le t^{-\frac{1}{2}} \end{equation} 
for all sufficiently small values of $t > 0$.
It follows from Fubini's theorem applied to the ``area under the graph" of the function $\mathcal{I}_G^{-\frac{\delta}{2}}$ regarded as a subset of $\Subd{G} \times \RR_{\ge 0}$ that
\begin{equation}
\label{eq:Fubini}
 \int_Z \mathcal{I}_G^{-\frac{\delta}{2}}(\Gamma) \; \mathrm{d} \nu(\Gamma) = \int_{\left[0,1\right]}  D(t) \; \mathrm{d} \lambda(t)
 \end{equation}
where $\lambda$ is the Lebesgue measure.  As the function $ D$ is monotone non-decreasing and the integral $\int_0^1 x^{-\frac{1}{2}} \mathrm{d}x$ converges, 
we conclude from Equation (\ref{eq:Fubini})   that the function $\mathcal{I}_G^{-\frac{\delta}{2}}$ is indeed $\nu$-integrable.
 \end{proof}

A careful examination of \S\ref{sec:nilpotent} and \S\ref{sec:Zass} shows that all of our arguments and proofs apply if the semisimple element $\sse = \sse(G)$ is replaced by its inverse $\sse(G)^{-1}$.
Consider therefore the Borel probability measure 
\begin{equation}
\label{eq:nu0}
 \hat{\mu} =  \eta_K * \frac{1}{2}(\delta_{\sse(G)} + \delta_{\sse(G)^{-1}}) * \eta_K = \frac{1}{2}(\mu_{\sse} + \mu_{\sse}^*)
 \end{equation}
on the   $k$-analytic group $G$. Clearly $\hat{\mu} = \hat{\mu}^*$ so that   the   averaging operator $A_{\hat{\mu}}$ acting on the Hilbert space  $  L^2(Z,\nu)$  is  self-adjoint.  

To ease our notations let $Z = \Subd{G}$ so that $\nu$ is a $\hat{\mu}$-stationary measure on the space $Z$. Recall that  $\mathcal{I}_G^{-\frac{\delta}{2}} \in L^1(Z,\nu)$ according to Proposition   \ref{prop:f_G^delta is integrable}. This implies that 
\begin{equation} F_G = \mathcal{I}_G^{-\frac{\delta}{4}} \in L^2(Z,\nu).\end{equation}
Since $0 <  \frac{\delta}{4} < \delta = \delta(G)$ it is possible to apply Proposition \ref{prop:from expansion to contraction} in a similar way to what was done in the proof of   Theorem \ref{thm:main theorem - local fields} but with respect to the smaller parameter $\frac{\delta}{4}$ and obtain an 
 analogue of the key inequality Equation   (\ref{eq:convolution with f_G}) with a ``worse" constant $ 0 < c < \hat{c} < 1$. A careful examination shows that there is a constant $\rho_1  > 0$ such that the sharper multiplicative inequality
\begin{equation}
\label{eq:multiplicative bound}
A_{\hat{\mu}} F_G \le \hat{c} F_G
\end{equation}
holds for $\nu$-almost every subgroup $\Gamma \in Z$ with $\mathcal{I}_G(\Gamma) < \rho_1$. This constant $\rho_1$ depends only on the $k$-analytic group $G$ and is independent of the measure $\nu$.


%
%
%

\begin{theorem}
\label{thm:uniform local spectral gap at infinity}
Let $Z_{<\rho_1}$ be the \emph{$\rho_1$-thin part} of the Chabauty space $\Subd{G}$, namely
$$ Z_{<\rho_1} = \{ \Gamma \in \Subd{G} \: : \: \mathcal{I}_G(\Gamma) < \rho_1 \}.$$
Let $\mathrm{P}: L^2(Z,\nu) \to L^2(Z_{<\rho_1},\nu)$ be the orthogonal projection to  $Z_{<\rho_1 } $. Then 
\begin{equation}
 \|\mathrm{P} A_{\hat{\mu}} \mathrm{P} _{| L^2(Z_{<\rho_0},\nu)} \| \le \hat{c}  
 \end{equation}
for some constant $0 < \hat{c} < 1$.
\end{theorem}

We emphasize that both constants $\rho_1$ and $\hat{c}$ depend \emph{only} on the   group $G$.
 

\begin{proof}[Proof of Theorem \ref{thm:uniform local spectral gap at infinity}]
Consider the spectrum   $S$ of the self-adjoint operator $\mathrm{P} A_{\hat{\mu}} \mathrm{P}$ restricted to the Hilbert subspace $L^2(Z_{<\rho_1},\nu)$. In particular $S$ is a compact subset of the real interval $\left[-1,1\right]$. Assume towards contradiction that the spectrum $S$ admits   a point $\lambda  \in S$ with $|\lambda | > \hat{c}$.

The spectral theorem for self-adjoint operators shows that $\mathrm{P} A_{\hat{\mu}} \mathrm{P}$ is conjugate to a multiplication operator. Therefore the Hilbert space $L^2(Z_{<\rho_1},\nu)$ has approximate eigenvfunctions for the value $\lambda$. Namely, there are vectors $v_n \in L^2(Z_{<\rho_1},\nu)$ with $\|v_n\| = 1$      that satisfy
 \begin{equation}
 \label{eq:approximate eigenvalue}
  \|(\mathrm{P} A_{\hat{\mu}} \mathrm{P}   - \lambda) v_n\| \xrightarrow{n\to\infty} 0.
  \end{equation}

Consider the    inner-products of the vector $F_G = \mathcal{I}_G^{-\frac{\delta}{4}} $ with the vectors $|v_n|$ in the Hilbert space $ L^2(Z,\nu)$.   The   fact that the operator $A_{\hat{\mu}}$ is self-adjoint gives
\begin{align*}
 \hat{c} \left<  F_G, |v_n| \right> &\ge \left< \mathrm{P} A_{\hat{\mu}} \mathrm{P} F_G, |v_n| \right> = 
 \left<  F_G,\mathrm{P} A_{\hat{\mu}} \mathrm{P}  |v_n| \right> \ge
 \\
&    \ge \left<  F_G, |\mathrm{P} A_{\hat{\mu}} \mathrm{P}   v_n| \right> =
 \left<  F_G, | \lambda v_n +  (\mathrm{P} A_{\hat{\mu}} \mathrm{P}   - \lambda) v_n | \right> \ge \\
  &\ge |\lambda|   \left<F_G, |v_n| \right>   -  \left<  F_G,  |(\mathrm{P} A_{\hat{\mu}} \mathrm{P}   - \lambda) v_n | \right>  \ge \\
  &= |\lambda|   \left<F_G, |v_n| \right>  -  \|  F_G\| \|(\mathrm{P} A_{\hat{\mu}} \mathrm{P}   - \lambda) v_n\|.
     \end{align*}
This stands in  a contradiction to Equation (\ref{eq:approximate eigenvalue}).
\end{proof}


We are ready to conclude  the proof of Theorem \ref{thm:weakly cocompact} stated in the introduction.  
\begin{cor}
\label{cor:lattices are weakly cocompact}
Every lattice $\Gamma$ in the group $G$ is weakly cocompact, i.e. $L^2(G/\Gamma)$ has spectral gap.
\end{cor}
\begin{proof}
Let $\Gamma  $  be a lattice in the $k$-analytic group $G$. Let $D = (G/\Gamma)_{\ge \rho}$ be the $\rho$-thick part of  the homogenous $G$-space $G/\Gamma$. In particular   $D$ is a compact subset. It follows from Theorem \ref{thm:uniform local spectral gap at infinity} that 
$$  \|\mathrm{P} A_{\hat{\mu}} \mathrm{P} _{| L^2((G/\Gamma)_{<\rho} )} \| \  \le \hat{c} < 1. $$
 Therefore the lattice $\Gamma$ is weakly cocompact by   Lemma \ref{lem:local spectral gap at infinity implies spectral gap}. \end{proof}

\appendix

\section{Maximally compact Cartan subalgebras}

Let $\mathfrak{g}_0$ be a real semisimple Lie algebra with Cartan involution $\theta$. Let   $\mathfrak{g}_0 = \mathfrak{l}_0 \oplus \mathfrak{p}_0$ be the corresponding Cartan decomposition. 
Let $\mathfrak{t}_0$ be a maximal abelian Lie subalgebra of $\mathfrak{l}_0$. Then $\mathfrak{h}_0 = Z_{\mathfrak{g}_0}(\mathfrak{t}_0)$  is a maximally compact  $\theta$-stable Cartan subalgebra of $\mathfrak{g}_0$. It satisfies  $\mathfrak{h}_0 = \mathfrak{t}_0 \oplus \mathfrak{a}_0$ where $\mathfrak{a}_0 \le \mathfrak{p}_0$ is some abelian Lie subalgebra (not necessarily maximal).

 We will use
  $\mathfrak{g},\mathfrak{h},\mathfrak{t},\mathfrak{l}$ and $\mathfrak{p}$ to denote the suitable complexifications. In particular $\mathfrak{g}$ is a complex semisimple Lie algebra with Cartan subalgebra $\mathfrak{h}$. Moreover $\mathfrak{g} = \mathfrak{l} \oplus \mathfrak{p}$. Let $\Delta = \Delta(\mathfrak{g}, \mathfrak{h})$ be the absolute root system. The roots $\Delta$  are real valued  on the real form $\mathfrak{h}_\RR = i\mathfrak{t}_0 \oplus \mathfrak{a}_0$ of   $\mathfrak{h}$. We may write $\mathfrak{g} = \mathfrak{h} \oplus \bigoplus_{\alpha \in \Delta} \mathfrak{g}_\alpha$ where each root space satisfies $\dim_\CC \mathfrak{g}_\alpha = 1$.

The Lie algebra $\mathfrak{l}_0$ is compact. Therefore $\mathfrak{l}_0$ is reductive \cite[Corollary 4.25]{knapp2013lie}. Denote  $\mathfrak{c}_0 = Z(\mathfrak{l}_0)$ so that $\mathfrak{l}_0 = \mathfrak{c}_0 \oplus \left[\mathfrak{l}_0, \mathfrak{l}_0\right]$. Write $\mathfrak{t}_0 = \mathfrak{c} \oplus \mathfrak{t}_0' $ and let denote $\mathfrak{c}, \mathfrak{t}'$ be the complexifications of $\mathfrak{c}_0, \mathfrak{t}_0$, respectively. 
 Then $\mathfrak{t}'$ is Cartan subalgebra of the complex semisimple Lie algebra $\left[\mathfrak{l}, \mathfrak{l}\right]$. Let $\Psi = \Psi (\left[\mathfrak{l}, \mathfrak{l}\right], \mathfrak{t}')$ be the corresponding root system. We will regard each root $\beta \in \Psi$ to be defined on $\mathfrak{t}$ by extending it by zero on the center $\mathfrak{c}$. Note that the roots  $\Psi$ take real values on $i\mathfrak{t}_0$.

Let $\Lambda_\text{root}$   and $\Lambda _\text{algebraic}$  respectively   be  the root lattice (i.e. the $\ZZ$-span of $\Psi$) and the weight lattice (i.e. the set of algebraically integral forms) in $(\mathfrak{t}_0')^*$. We have $\Lambda _\text{root} \subset  \Lambda _\text{algebraic}$ and  $\left[\Lambda _\text{algebraic} : \Lambda _\text{root}\right]$ equals the determinant of the Cartan matrix of $\Psi$. 


Let $G$ be a real Lie group with maximal compact subgroup $K$ containing a compact torus $T$ with Lie algebras $\mathfrak{g}_0, \mathfrak{l}_0$ and $\mathfrak{t}_0$, respectively\footnote{It is not  possible in general to assume that $K$ is adjoint, even if   $G$ is.  E.g. the real semisimple Lie group $  \mathrm{SO}(3,4)$ is adjoint but its maximal compact subgroup $ \mathrm{SO}(3) \times \mathrm{SO}(4)$ is not. Therefore we are forced to allow the set of  analytical forms to be strictly larger than  the root lattice.}. Consider the lattice $\Lambda_\text{analytic}$ of analytic roots in $\mathfrak{t}_0^*$. It satisfies  $\Lambda _\text{root} \subset \Lambda_\text{analytic} \cap (\mathfrak{t}'_0)^* \subset \Lambda _\text{algebraic}$. By definition $\dim_\RR \mathfrak{t}_0  = \mathrm{rank}(K)$ so that $\Lambda_\text{analytic} \cong \ZZ^{\mathrm{rank}(K)}$.

The adjoint action of the compact torus $T$ on the   complexification $\mathfrak{g}$ admits a simultaneous eigenvalue decomposition  given by   restricting the  absolute roots $\Delta$ to $\mathfrak{t}_0$. Let $r : \mathfrak{h}^* \to \mathfrak{t}^*$ be the restriction map. 
For every root  $\alpha \in \Delta$ we have
$$ \Ad{\exp(X)} Y = r( \alpha) (X) Y \quad \forall X \in \mathfrak{t}_0, Y \in \mathfrak{g}_\alpha.$$
We conclude that $r(\alpha)  \in \Lambda_\text{analytic}$ for every root $\alpha \in \Delta$.


Generally speaking,  a root $\alpha \in \Delta$ is called real (imaginary resp.) if $\alpha(\mathfrak{h}_0)$ belongs to $\RR$ ($i\RR$ resp.). Otherwise $\alpha$ is called complex. In other words  $\alpha$ is real if and only if $\alpha$ vanishes on $\mathfrak{t}_0$. Likewise $\alpha$ is imaginary if and only if $\alpha$ vanishes on $\mathfrak{a}_0$.

Since the Cartan subalgebra $\mathfrak{h}_0$ is maximally compact $\Delta$ admits no real roots \cite[Proposition 6.70]{knapp2013lie}. Equivalently $r(\alpha) \neq 0$ for all $\alpha \in \Delta$. The Cartan involution $\theta$ determines an involution of $\Delta$ fixing pointwise the imaginary roots and permuting the complex roots in $2$-cycles.
Every imaginary root $\alpha \in \Delta$ can be classified as  being either compact if $\mathfrak{g}_\alpha \le \mathfrak{l}$ or non-compact if $\mathfrak{g}_\alpha \le \mathfrak{p}$ (one of these possibilities must occur since $\mathfrak{g}$ is $\theta$-stable). Moreover $\Psi \subset r(\Delta)$ for every root of $\Psi$ is the restriction of some compact imaginary root.

Fix a lexicographic notion of positivity  $\Delta^+$ taking  $\mathfrak{t}_0$ before $\mathfrak{a}_0$. Let $\Phi^+$ be the induced notion of positivity. It follows that $r(\alpha) > 0$ for every $\alpha \in \Delta^+$ (however $r(\alpha)$ need not be a root of $\Phi$ in general, just a positive analytic form).

Let $\Pi^+ \subset \Delta^+$ be the set of simple positive roots. Write 
$$\Pi^+ = \Pi^+_\text{complex} \cup  \Pi^+_\text{compact} \cup \Pi^+_\text{non-compact}$$ where $ \Pi^+_\text{complex},  \Pi^+_\text{compact}$ and $ \Pi^+_\text{non-compact}$ consist of the complex, imaginary compact and imaginary non-compact simple positive roots, respectively. Note that $\Pi^+_\text{complex} \neq \emptyset$  if and only if $\theta$ induces an non-trivial automorphism of $\Pi^+$.
By the Borel--de Siebenthal theorem we may assume without loss of generality that the notion of positivity has been chosen so that $|\Pi^+_\text{non-compact} | \le 1$, i.e. there is at most one imaginary non-compact root. See \cite{borel1949sous} and \cite[Theorem 6.96]{knapp2013lie}.

\begin{prop}
\label{prop:restricting to maximally compact}
The restriction of the simple positive absolute roots $\Pi^+$  to the maximal compact abelian Lie subalgebra $\mathfrak{t}_0$ spans a lattice $\Lambda$ with $\Lambda_\text{root} \le \Lambda \le \Lambda_\text{analytic}$ and $\left[\Lambda_\text{analytic}:\Lambda\right] < \infty$.
\end{prop}
Note that the second conclusion does not automatically  follow from the first, for $\left[\Lambda_\text{analytic}:\Lambda_\text{root}\right] < \infty$ only provided $\dim \mathfrak{c}_0 = 0$.
\begin{proof}[Proof of Proposition \ref{prop:restricting to maximally compact}]
The Lie subalgebra $\mathfrak{c}_0 = Z(\mathfrak{l}_0)$ satisfies  $  \dim_\RR \mathfrak{c}_0 \le 1$ \cite[Appendix C.3]{knapp2013lie}. If $\mathfrak{c}_0$ is non-zero let $\gamma \in \Lambda_\text{analytic}$ denote  the generator of the infinite cyclic group of analytic integral forms on $\mathfrak{c}_0^*$.


We now consider the restrictions $r(\alpha) = \alpha_{|\mathfrak{t}_0}$ of the various types of simple positive roots $\alpha \in \Pi^+$:

\begin{itemize}
\item 
Let $\alpha \in \Pi^+_\text{compact}$ be a \underline{compact imaginary root}. Then $\mathfrak{g}_\alpha \le \mathfrak{l}$ which means that $r(\alpha) \in \Phi^+$. We claim that $r(\alpha)$ is simple. Indeed if $r(\alpha) = \beta_1 + \beta_2$ for some pair of positive roots $\beta_1,\beta_2 \in \Phi^+$ then $\beta_i = r(\alpha_i)$ for some compact imaginary $\alpha_i \in \Delta^+$ so that  $\alpha = \alpha_1 +  \alpha_2$,  a contradiction.

\item Let $\alpha \in \Pi^+_\text{complex}$ be a \underline{complex root}. Then  $\beta = r(\frac{\alpha + \theta \alpha}{2}) \in \Phi^+$. We claim that $\beta$ is simple. Indeed if $\beta = \beta_1 + \beta_2$ for some pair of positive roots $\beta_1,\beta_2 \in \Phi^+$ then $\beta_i = r(\alpha_i)$ for some compact imaginary $\alpha_i \in \Delta^+$. Then $2 \alpha_1 + 2\alpha_2 = \alpha + \theta \alpha$,  a contradiction as both $\alpha$ and $\theta \alpha$ are simple and not imaginary. 

\item Let $\alpha \in \Pi^+_\text{non-compact}$ be a \underline{non-compact imaginary root}. 
\begin{itemize}
\item Assume  $\dim \mathfrak{c}_0 = 0$. Let $\alpha' \in \Delta^+$ be the smallest complex root such that $\alpha < \alpha'$ if $\Pi^+_\text{complex} \neq \emptyset$   and $2\alpha < \alpha'$ otherwise. Write $\alpha' = \sum_{\lambda \in \Pi^+} n_{\lambda} \lambda$ so that $n_\lambda \ge 0$. In addition $n_{\alpha} \ge 1$ or $n_{\alpha}  \ge 2$ in the first and second cases, respectively. Then $\beta = r(\frac{\alpha' + \theta \alpha'}{2}) \in \Phi^+$ is simple \cite[Appendix C.3]{knapp2013lie}.   Rearranging gives
$$ r(\alpha)  \in \frac{1}{n_\alpha}\beta + \sum_{\alpha \in \Pi^+_\text{complex} \cup \Pi^+_\text{compact}} \ZZ r(\alpha).$$
All the roots of $\Phi^+$ that appear on the right hand side are simple.
\item Assume $\dim \mathfrak{c}_0 = 1$. In particular $\Pi^+_\text{complex} = \emptyset$  so that the other   simple roots of $\Pi^+$ are all compact imaginary \cite[Appendix C.3]{knapp2013lie} and restrict to simple roots of $\Phi^+$. As $\dim \mathfrak{t}'_0 = \dim\mathfrak{t}_0 - 1 = |\Pi^+| - 1 = |\Pi^+_\text{compact}|$ the  roots $r(\Pi^+_\text{compact})$ span $(\mathfrak{t}'_0)^*$.   It follows that the restriction $r(\alpha)$ is non-trivial on $\mathfrak{c}_0$.
\end{itemize}
\end{itemize}

Let   $\Lambda \le \Lambda_\text{analytic}$ be the lattice spanned by   $ r(\Phi^+)$. If $\Pi^+_\text{non-compact} = \emptyset$ then $ r(\Phi^+)$ coincides with set of simple positive roots for $\Phi^+$ and so  $\Lambda = \Lambda_\text{root}$. In the general case  the above discussion shows that $ r(\Phi^+)$ is a basis of $\Lambda$, $\Lambda_\text{root} \le \Lambda \le \Lambda_\text{analytic}$ and $\left[\Lambda_\text{analytic}:\Lambda\right] < \infty$, as required.
\end{proof}

\bibliographystyle{alpha}
\bibliography{QWUD}

\end{document}